\documentclass[12pt]{article}
\usepackage{makeidx}
\usepackage[combine]{ucs}
\usepackage[utf8x]{inputenc}
\usepackage{epsfig,mathrsfs}
\usepackage{lmodern,amsmath,amsthm,amssymb,graphicx,subfigure,xcolor}
\usepackage{pslatex,xspace}
\usepackage[a4paper]{geometry}
\usepackage{authblk}

\usepackage{rotating}
\usepackage{amstext}
\usepackage{amsopn,enumitem}
\usepackage{tikz}
\usetikzlibrary{shapes,arrows,decorations.markings,calc}
\usepackage[pdftex,                %
    bookmarks         = true,
    bookmarksnumbered = true,
    pdfpagemode       = none,
    pdfstartview      = FitH,
    pdfpagelayout     = SinglePage,
    colorlinks        = true,
    urlcolor          = magenta,
]{hyperref}
\newlist{exoenum}{enumerate}{3}
\setlist[exoenum,1]{label=\arabic*)}
\setlist[exoenum,2]{label=\alph*)}
\setlist[exoenum,3]{label=\roman*)}

\newcommand{\ra}{\rightarrow}
\newcommand{\la}{\leftarrow}

\renewcommand\geq{\geqslant}
\renewcommand\leq{\leqslant}

\renewcommand\le{\leqslant}

\theoremstyle{plain}
\newtheorem{theorem}{Theorem}

\newtheorem{proposition}[theorem]{Proposition}
\newtheorem{claim}{Claim}[theorem]

\newtheorem{corollary}[theorem]{Corollary}
\newtheorem{lemma}[theorem]{Lemma}

\theoremstyle{definition}

\newtheorem{problem}[theorem]{Problem}
\newtheorem{conjecture}[theorem]{Conjecture}

\newenvironment{subproof}{\par\noindent {\it Subproof}.\ }{\hfill$\lozenge$\par\vspace{11pt}}

\DeclareMathOperator{\spa}{span}

\DeclareMathOperator{\dist}{dist}

\DeclareMathOperator{\Forb}{Forb}
\DeclareMathOperator{\SForb}{S-Forb}

\DeclareMathOperator{\Dig}{Dig}
\DeclareMathOperator{\lvl}{lvl}

\def\borne{(k+\ell -2)(k+\ell-3)(2\ell+2)(k+\ell+1)}

\title{\bf Subdivisions of oriented cycles in digraphs with large chromatic number\footnote{This work was supported by ANR under contract STINT ANR-13-BS02-0007.}}
\author[1]{Nathann Cohen}
\author[2,3]{Fr\'ed\'eric Havet}
\author[2,3,4]{William Lochet}
\author[3,2]{Nicolas Nisse}
\affil[1]{ CNRS, LRI, Univ. Paris Sud, Orsay, France}
\affil[2]{ Univ. Nice Sophia Antipolis, CNRS, I3S, UMR 7271, 06900 Sophia Antipolis, France}
\affil[3]{ INRIA, France}
\affil[4]{ LIP, ENS de Lyon, France}

\begin{document}

\maketitle

\begin{abstract}
An {\it oriented cycle} is an orientation of a undirected cycle.
We first show that for any oriented cycle $C$, there are digraphs containing no subdivision of $C$  (as a subdigraph) and arbitrarily large chromatic number.
In contrast, we show that for any $C$ a cycle with two blocks, every strongly connected digraph with sufficiently large chromatic number contains a subdivision of $C$. We prove a similar result for the antidirected cycle on four vertices (in which two vertices have out-degree $2$ and two vertices have in-degree $2$). 
\end{abstract}

\section{Introduction}

What can we say about the subgraphs of a graph $G$ with large chromatic
number? Of course, one way for a graph to have large chromatic number is
to contain a large complete subgraph. However, if we consider graphs with large
chromatic number and small clique number, then we can ask what other subgraphs
must occur. We can avoid any graph $H$ that contains a cycle because, as proved by Erd\H{o}s~\cite{Erd59},  there are graphs with arbitrarily high girth and
chromatic number. Reciprocally, one can easily show that  every $n$-chromatic graph contains every tree of order $n$ as a subgraph.

The following more general question attracted lots of attention.
\begin{problem}\label{prob-undirected}
Which are the graph classes ${\cal G}$ such that every graph with sufficiently large chromatic number contains an element of ${\cal G}$ ?
\end{problem}
If such a class is finite, then it must contain a tree, by the above-mentioned result of Erd\H{o}s.
If it is infinite however, it does not necessary contains a tree.
For example, every graph with chromatic number at least $3$ contains an odd cycle. This was strengthened by Erd\H{o}s and Hajnal~\cite{ErHa66} who proved that every graph with chromatic number at least $k$ contains an odd cycle of length at least $k$.
A counterpart of this theorem for even length was settled by Mih\'ok and Schiermeyer~\cite{MiSc04}: every graph with chromatic number at least $k$ contains an even cycle of length at least $k$.
Further results on graphs with prescribed lengths of cycles have been obtained~\cite{Gya92,MiSc04,Wan08,LRS10,KRS11}.

\medskip
In this paper, we  consider the analogous problem for directed graphs, which is in fact a generalization of the undirected one.
 The {\it chromatic number} $\chi(D)$ of a digraph $D$ is the chromatic number of its underlying graph.
The {\it chromatic number} of a class of digraphs ${\cal D}$, denoted by $\chi({\cal D})$, is the smallest $k$ such that
$\chi(D)\leq k$ for all $D\in {\cal D}$, or $+\infty$ if no such $k$ exists. By convention, if  ${\cal D}=\emptyset$, then $\chi({\cal D})=0$. If $\chi({\cal D}) \neq +\infty$, we say that ${\cal D}$ {\it has bounded chromatic number}.

We are interested in the following question~: which are the  digraph classes ${\cal D}$ such that every digraph with sufficiently large chromatic number contains an element of ${\cal D}$ ?
Let us denote by  $\Forb(H)$ (resp. $\Forb({\cal H})$) the class of digraphs that do not contain $H$ (resp. any element of ${\cal H}$) as a subdigraph.
The above question can be restated as follows :
\begin{problem}
Which are the classes of digraphs ${\cal D}$ such that $\chi (\Forb({\cal D})) < +\infty$ ?
\end{problem}

 This is a generalization of Problem~\ref{prob-undirected}. Indeed,
 let us denote by $\Dig({\cal G})$ the set of digraphs whose underlying digraph is in ${\cal G}$;
 Clearly, $\chi({\cal G})=\chi(\Dig({\cal G}))$.

An {\it oriented graph} is an orientation of a (simple) graph; equivalently it is a digraph with no directed cycles of length $2$.
Similarly, an {\it oriented path} (resp. {\it oriented cycle}, {\it oriented tree}) is an orientation of a path (resp. cycle, tree).
An oriented path (resp., an oriented cycle) is said {\it directed} if all nodes have in-degree and out-degree at most $1$. 

Observe that if $D$ is an orientation of a graph $G$ and $\Forb(D)$ has bounded chromatic number, then $\Forb(G)$ has also bounded chromatic number, so $G$ must be a tree.
Burr proved that every $(k-1)^2$-chromatic digraph contains every oriented tree of order $k$.
This was slightly improved by Addario-Berry et al.~\cite{AHS+13} who proved the following.
\begin{theorem}[Addario-Berry et al.~\cite{AHS+13}]\label{thm:univ}
 Every $(k^2/2-k/2+1)$-chromatic oriented graph contains every oriented tree of order $k$. In other words, for every oriented tree $T$ of order $k$, $\chi(\Forb(T))\leq k^2/2-k/2$.
\end{theorem}

 \begin{conjecture}[Burr~\cite{Burr80}]\label{cnj:tn}
	  Every $(2k-2)$-chromatic digraph $D$ contains a copy of any oriented tree $T$ of order $k$.
	\end{conjecture}

For special oriented trees $T$, better bounds on the chromatic number of $\Forb(T)$ are known. The most famous one, known as Gallai-Roy Theorem, deals with directed paths (a {\it directed path} is an oriented path in which all arcs are in the same direction) and can be restated as follows, denoting by $P^+(k)$ the directed path of length $k$.

  \begin{theorem}[Gallai~\cite{Gal68}, Hasse~\cite{Has64}, Roy~\cite{Roy67}, Vitaver~\cite{Vit62}]\label{thm:gallairoy}
	  $\chi(\Forb(P^+(k)))=k$.
	\end{theorem}

The chromatic number of the class of digraphs not containing a prescribed  oriented path with two blocks  ({\it blocks} are maximal directed subpaths) has been determined by Addario-Berry et al.~\cite{AHT07}.

\begin{theorem}[Addario-Berry et al.~\cite{AHT07}]\label{thm:2blocks}
Let $P$ be an oriented path with two blocks on $n$ vertices.
\begin{itemize}
\item If $n=3$, then $\chi(\Forb(P))=3$.
\item If $n\geq 4$, then $\chi(\Forb(P))=n-1$.
\end{itemize}
\end{theorem}

In this paper, we are interested in the chromatic number of $\Forb({\cal H})$ when ${\cal H}$ is an infinite family of oriented cycles.
Let us denote by $\SForb(D)$ (resp. $\SForb({\cal D})$) the class of digraphs that contain no subdivision of $D$ (resp. any element of ${\cal D}$) as a subdigraph.
We are particularly interested in the chromatic number of $\SForb({\cal C})$, where ${\cal C}$ is a family of oriented cycles.



Let us denote by $\vec{C}_k$ the directed cycle of length $k$.
For all $k$,  $\chi(\SForb(\vec{C}_k))=+\infty$ because transitive tournaments have no directed cycle.
 Let us denote by $C(k,\ell)$ the oriented cycle with two blocks, one of length $k$ and the other of length $\ell$. Observe that the oriented cycles with two blocks are the subdivisions of $C(1,1)$. As pointed Gy\'arf\'as and Thomassen (see \cite{AHT07}), there are acyclic oriented graphs with arbitrarily large chromatic number and no oriented cycles with two blocks. Therefore $\chi(\SForb(C(k,\ell))) =+\infty$.
We first generalize these two results to every oriented cycle.

\begin{theorem}\label{thm:infty1}
For any oriented cycle $C$,
$$\chi(\SForb(C))=+\infty.$$
\end{theorem}

In fact, we show a stronger theorem (Theorem~\ref{thm:blocks}): for any positive integer $b$, there are digraphs of arbitrarily high chromatic number that contains no oriented cycles with less than $b$ blocks. It directly implies the following generalization of the previous theorem.

\begin{theorem}\label{thm:infty2}
For any finite family ${\cal C}$ of oriented cycles,
$$\chi(\SForb({\cal C}))=+\infty.$$
\end{theorem}

In contrast, if ${\cal C}$ is an infinite family of oriented cycles, $\SForb({\cal C})$ may have bounded chromatic number. By the above argument, such a family must contain a cycle with at least $b$ blocks for every positive integer $b$.
A cycle $C$ is {\it antidirected} if any vertex of $C$ has either in-degree $2$ or out-degree $2$ in $C$. In other words, it is an oriented cycle in which all blocks have length $1$.
Let us denote by ${\cal A}_{\geq 2k}$ the family of antidirected cycles of length at least $2k$.
In Theorem~\ref{thm:antidirected}, we prove that $\chi(\Forb({\cal A}_{\geq 2k}))\leq 8k-8$.
Hence we are left with the following problem.

\begin{problem}\label{prob:infini}
What are the infinite families of oriented cycles ${\cal C}$ such that $\Forb({\cal C})<+\infty$ ?\\
What are the infinite families of oriented cycles ${\cal C}$ such that $\SForb({\cal C})<+\infty$ ?
\end{problem}

On the other hand, considering strongly connected (strong for short) digraphs may lead to dramatically different result.
An example is provided by the following celebrated result due to Bondy~\cite{Bon76} : {\it  every strong digraph of chromatic number at least $k$ contains a directed cycle of length at least $k$.}
Denoting the class of strong digraphs by ${\cal S}$, this result can be rephrased as follows.

\begin{theorem}[Bondy~\cite{Bon76}]\label{thm:bondy}
$\chi(\SForb(\vec{C}_k)\cap {\cal S}) =k-1$.
 \end{theorem}

Inspired by this theorem, Addario-Berry et al.~\cite{AHT07} posed the following problem.
\begin{problem}
Let $k$ and $\ell$ be two positive integers. Does
$\SForb(C(k,\ell)\cap {\cal S})$ have bounded chromatic number?
\end{problem}
In Subsection~\ref{subsec:2blocks-gen}, we answer to this problem in the affirmative.
In Theorem~\ref{thm:Ckk} we prove 
$$\chi(\SForb(C(k,\ell)\cap {\cal S}) \leq \borne, \mbox{ for all }  k\geq \ell\geq 2, k\geq 3.$$

Note that since $\chi(\SForb(C(k',\ell')\cap {\cal S}) \leq \chi(\SForb(C(k,\ell)\cap {\cal S})$ if $k'\leq k$ and $\ell'\leq \ell$, this gives also an upper bound when $k$ or $\ell$ are small. However, in those cases, we prove better upper bounds.
In Corollary~\ref{cor:k1}, we prove 
$$\chi(\SForb(C(k,1))\cap {\cal S})\leq \max\{k+1, 2k-4\} \mbox{ for all } k.$$
We also give in Subsection~\ref{subsec:2blocks-gen} the exact value of $\SForb(C(k,\ell)\cap {\cal S})$ for $(k,\ell)\in \{(1,2), (1,3), (2,3)\}$.


\medskip

More generally, one may wonder what happens for other oriented cycles.
\begin{problem}
Let $C$ be an oriented cycle with at least four blocks.
Is $\chi(\SForb(C)\cap {\cal S})$ bounded?
\end{problem}

In Section~\ref{sec:4block}, we show that $\chi(\SForb(\hat{C}_4)\cap {\cal S}) \leq 24$ where $\hat{C}_4$ is the antidirected cycle of order $4$.

\section{Definitions}

We follow~\cite{BoMu08} for basic notions and notations.
Let $D$ be a digraph. $V(D)$ denotes its vertex-set and $A(D)$ its arc-set.

If $uv \in A(D)$ is an arc, we sometimes write $u\ra v$ or $v\la u$.

For any $v \in V(D)$, $d^+(v)$ (resp. $d^-(v)$) denotes the out-degree (resp. in-degree) of $v$.  $\delta^+(D)$ (resp. $\delta^-(D)$) denotes the minimum out-degree (resp. in-degree) of $D$.


 An \emph{oriented path} is any orientation of a \emph{path}. The \emph{length} of a path is the number
	of its arcs. Let $P = (v_1, \ldots , v_n)$ be an oriented path. If $v_iv_{i+1} \in A(D)$, then $v_i
	v_{i+1}$ is a \emph{forward arc}; otherwise, $v_{i+1}v_i$ is a \emph{backward arc}. $P$ is a \emph{directed
	path} if all of its arcs are either forward or backward ones. For convenience, a directed path with forward arcs only is called a {\it dipath}.
	A \emph{block} of $P$ is a maximal
	directed subpath of $P$. A path is entirely determined by the sequence $(b_1, \dots , b_p)$ of the lengths of its blocks and the sign $+$ or $-$ indicating if the first arc is forward or backward respectively. Therefore we denote by $P^+(b_1, \ldots, b_p)$ (resp. $P^-(b_1, \ldots, b_p)$) an oriented path whose first arc is forward (resp. backward) with $p$ blocks, such that the $i$th block along it has length $b_i$.

	Let $P=(x_1,x_2,\dots ,x_n)$ be an oriented path.
We say that $P$ is an {\it $(x_1,x_n)$-path}.
For every $1 \leq i \leq j \leq n$, we note $P[x_i,x_j]$ (resp. $P]x_i,x_j[$, $P[x_i,x_j[$, $P]x_i,x_j]$) the oriented subpath $(x_i, \dots, x_j)$ (resp.
$(x_{i+1}, \dots, x_{j-1})$, $(x_{i}, \dots, x_{j-1})$, $(x_{i+1}, \dots, x_{j})$).

The vertex $x_1$ is the {\it initial vertex} of $P$ and $x_n$ its {\it terminal vertex}.
Let $P_1$ be an $(x_1,x_2)$-dipath and $P_2$ an $(x_2,x_3)$-dipath which are disjoint except in $x_2$.
Then $P_1 \odot P_2$ denotes the $(x_1,x_3)$-dipath obtained from the concatenation of these dipaths.

The above definitions and notations can also be used for oriented cycles. Since a cycle has no initial and terminal vertex, we have to choose one as well as a direction to run through the cycle.Therefore if $C=(x_1,x_2, \dots , x_n,x_1)$ is an oriented cycle, we always assume that $x_1x_2$ is an arc, and if $C$ is not directed that $x_1x_n$ is also an arc.

A path or a cycle (not necessarily directed) is {\it Hamiltonian} in a digraph if it goes through all vertices of $D$.


The digraph $D$ is {\it connected} (resp.\ {\it $k$-connected}) if its underlying graph is connected (resp.\ $k$-connected).
It is {\it strongly connected}, or {\it strong}, if for any two vertices $u,v$, there is a $(u,v)$-dipath in $D$.
It is {\it k-strongly connected} or {\it k-strong}, if for any set $S$ of $k-1$ vertices $D - S$ is strong.
A {\it strong component} of a digraph is an inclusionwise maximal strong subdigraph.
Similarly, a {\it $k$-connected component} of a digraph is an inclusionwise maximal $k$-connected subdigraph.


\section{Antidirected cycles}\label{sec:antidirect}

The aim of this section is to prove the following theorem, that establish that $\chi(\Forb({\cal A}_{\geq 2k}))\leq 8k-8$.
\begin{theorem}\label{thm:antidirected}
Let $D$ be an oriented graph and $k$ an integer greater than $1$.\\
If $\chi(D)\geq 8k-7$, then $D$ contains an antidirected cycle of length at least $2k$.
\end{theorem}

A graph $G$ is {\it $k$-critical} if $\chi(G)=k$ and $\chi(H)<k$ for any proper subgraph $H$ of $G$.
Every graph with chromatic number $k$ contains a $k$-critical graph.
We denote by $\delta(G)$ the minimum degree of the graph $G$.
The following easy result is well-known.
	\begin{proposition}\label{prop:dkcrit}
	  If $G$ is a $k$-critical graph, then $\delta(G) \geq k - 1$.
	\end{proposition}

Let $(A,B)$ be a bipartition of the vertex set of a digraph $D$.
We denote by $E(A,B)$ the set of arcs with tail in $A$ and head in $B$ and by $e(A,B)$ its cardinality.

\begin{lemma}[Burr~\cite{Bur82}]\label{maxcut}\rm
Every digraph $D$ contains a partition $(A,B)$ such that
$e(A,B)\geq |E(D)|/4$.
\end{lemma}

\begin{lemma}[Burr~\cite{Bur82}]\label{degremin}
Let $G$ be a bipartite graph and $p$ be an integer.
If $|E(G)|\geq p|V(G)|$, then $G$ has a subgraph with minimum degree at least $p+1$.
\end{lemma}

\begin{lemma}\label{lemma:long-cycle}
Let $k\geq 1$ be an integer.
Every bipartite graph with minimum degree $k$ contains a cycle of order at least $2k$.
\end{lemma}
\begin{proof}
Let $G$ be a bipartite graph with bipartition $(A,B)$.
Consider a longest path $P$ in $G$.
Without loss of generality, we may assume that one of its ends $a$ is in $A$.
All neighbours of $a$ are in $P$ (otherwise $P$ can be lengthened). Let $b$ be the furthest neighbour of $a$ in $B$ along $P$.
Then $C=P[a,b]\cup ab$ is a cycle containing at least $k$ vertices in $B$, namely the neighbours of $a$.
Hence $C$ has length at least $2k$, since $G$ is bipartite.
\end{proof}

\begin{proof}[Proof of Theorem~\ref{thm:antidirected}]
It suffices to prove that every $(8k-7)$-critical oriented graph contains an antidirected cycle of length at least $2k$.

Let $D$ be a $(8k-7)$-critical oriented graph.
By Proposition~\ref{prop:dkcrit}, it has minimum degree at least $8k-8$, so $|E(D)|\geq (4k-4) |V(D)|$.
By Lemma~\ref{maxcut}, $D$ contains a partition such that $e(A,B)\geq |E(D)|/4\geq (k-1) |V(D)|$.
Consequently, by Lemma~\ref{degremin}, there are two sets $A'\subseteq A$ and $B'\subseteq B$ such that every vertex in $A'$ (resp. $B'$) has at least $k$ out-neighbours in $B'$ (resp. $k$ in-neighbours in $A'$).
Therefore, by Lemma~\ref{lemma:long-cycle}, the bipartite oriented graph induced by $E(A',B')$ contains a cycle of length at least $2k$, which is necessarily antidirected.
\end{proof}

\begin{problem}
Let $\ell$ be an even integer.
What the minimum integer $a(\ell)$ such that every oriented graph with chromatic number at least $a(\ell)$ contains an antidirected cycle of length at least $\ell$ ?
\end{problem}

\section{Acyclic digraphs without cycles with few blocks}

The aim of this section is to establish Theorems~\ref{thm:infty1} and~\ref{thm:infty2}.
To do so we will use a result on hypergraph colouring.

\medskip

A {\it cycle} of {\it length} $k\geq 2$ in a hypergraph ${\cal H}$ is an
alternating cyclic sequence $e_0, v_0, e_1, v_1,\dots$ $ e_{k-1}, v_{k-1}, e_0$ of
distinct hyperedges and vertices in ${\cal H}$ such that $v_i\in e_i\cap
e_{i+1}$ for all $i$ modulo $k$.  The {\it girth} of a hypergraph is the length
of a shortest cycle.

A hypergraph $\mathcal H$ on a ground set $X$ is said to be {\it weakly
  $c$-colourable} if there exists a colouring of the elements of $X$ with $c$
colours such that no hyperedge of $\mathcal H$ is monochromatic. The {\it weak
  chromatic number} of $\mathcal H$ is the least $c$ such that $\mathcal H$ is
weakly $c$-colourable. Erd{\H{o}}s and Lov{\'a}sz~\cite{EL75} (and more recently
Alon {\em et al.}\cite{alon2014coloring}) proved the following result:

\begin{theorem}\cite[Theorem 1']{EL75}, \cite{alon2014coloring}
  \label{thm:alonisgreat}
  For $k, g, c\in\mathbb{N}$, there exists a $k$-uniform hypergraph with girth
  larger than $g$ and weak chromatic number larger than $c$.
\end{theorem}

Our construction relies on the hypergraphs whose existence is established by Theorem~\ref{thm:alonisgreat}.

\begin{theorem}\label{thm:blocks} For any positive integers $b,c$, there
  exists an acyclic digraph $D$ with $\chi(D)\geq c$ in which all oriented cycles
  have more than $b$ blocks.
\end{theorem}
\begin{proof}
  We shall prove the result by induction on $c$, the result holding trivially
  for $c=2$ with $D$ the directed path on two vertices. We thus assume our
  claim to hold for a graph $D_c$ with $\chi(D_c)=c$, and show how extend it to
  $c+1$.

  Let $p$ be the number of proper $c$-colourings of $D_c$, and let those
  colourings be denoted by $col^1_c,...,col^{p}_c$. By
  Theorem~\ref{thm:alonisgreat} there exists a $c\times p$-uniform hypergraph
  $\mathcal H$ with weak chromatic number $>p$ and girth $>b/2$.  Let $X=\{x_1, \dots
  ,x_n\}$ be the ground set of ${\cal H}$.

  We construct $D_{c+1}$ from $n$ disjoint copies $D_c^1,...,D_c^n$ of $D_c$ as
  follows. For each hyperedge $S\in \mathcal H$, we do the following (see
  Figure~\ref{fig:hyperedge}) :

\begin{itemize}
\item We partition $S$ into $p$  sets $S_1,\dots,S_{p}$ of cardinality $c$.\\[-6mm]
\item For each set $S_i=\{x_{k_1},\dots, x_{k_c}\}$, we choose vertices
  $v_{k_1}\in D^{k_1}_c,\dots ,v_{k_c}\in D^{k_c}_c$ such that
  $col^{i}_c(v_{k_1})=1,\dots ,col^{i}_c(v_{k_c})=c$, and add a new vertex
  $w_{S,i}$ with $v_{k_1},\dots ,v_{k_c}$ as in-neighbours.
\end{itemize}

Let us denote by $W$ the set of vertices of $D_{c+1}$ that do not belong to any
of the copies of $D_c$ (i.e. the $w_{S,i}$). We now prove that the resulting
digraph $D_{c+1}$ is our desired digraph.

\begin{figure}[hbtp]
\begin{center}
  \begin{tikzpicture}[scale=1.4]
    \draw[rounded corners=.4cm] (-4,-.4) rectangle (4,.4);
    \foreach \i in {1,2,-1,-2} {
      \draw[thick,gray,dashed] (\i,.37) -- (\i,-.4);
    }
    \draw[gray] node at (-1.475,0) {\Large $\cdots$};
    \draw[gray] node at (+1.525,0) {\Large $\cdots$};

    \draw node at (-3,-.7) {\small $S_1$};
    \draw node at ( 0,-.7)   {\small $S_i$};
    \draw node at (+3,-.7) {\small $S_{p}$};

    \draw node at (1,.6) {\small $S\in \mathcal H$};

    \begin{scope}[xshift=-3cm]
      \foreach \j in {-1,...,1} {
        \begin{scope}[xshift=-.6*\j cm]
          \draw[red] node[circle,fill,scale=.3] (g\j0) at (0*360/5+90:.2cm) {};
          \draw[green] node[circle,fill,scale=.3] (g\j1) at (1*360/5+90:.2cm) {};
          \draw[blue] node[circle,fill,scale=.3] (g\j2) at (2*360/5+90:.2cm) {};
          \draw[red] node[circle,fill,scale=.3] (g\j3) at (3*360/5+90:.2cm) {};
          \draw[blue] node[circle,fill,scale=.3] (g\j4) at (4*360/5+90:.2cm) {};
          \draw[gray,->] (g\j1) -- (g\j0);
          \draw[gray,->] (g\j2) -- (g\j1);
          \draw[gray,->] (g\j2) -- (g\j3);
          \draw[gray,->] (g\j3) -- (g\j4);
          \draw[gray,->] (g\j4) -- (g\j0);
        \end{scope}
      }
      \draw node[circle,fill,scale=.3,label=0:$w_{S,1}$] (g0) at (0,1) {};
      \draw[->] (g-13) .. controls +(0,1) and +(0,-.5) .. (g0);
      \draw[->] (g04) .. controls +(0,.7) and +(0,-.5) .. (g0);
      \draw[->] (g11) .. controls +(0,.7) and +(0,-.5) .. (g0);

    \end{scope}

    \begin{scope}[xshift=+3cm]
      \foreach \j in {-1,...,1} {
        \begin{scope}[xshift=-.6*\j cm]
          \draw[green] node[circle,fill,scale=.3] (f\j0) at (0*360/5+90:.2cm) {};
          \draw[blue] node[circle,fill,scale=.3] (f\j1) at (1*360/5+90:.2cm) {};
          \draw[green] node[circle,fill,scale=.3] (f\j2) at (2*360/5+90:.2cm) {};
          \draw[blue] node[circle,fill,scale=.3] (f\j3) at (3*360/5+90:.2cm) {};
          \draw[red] node[circle,fill,scale=.3] (f\j4) at (4*360/5+90:.2cm) {};
          \draw[gray,->] (f\j1) -- (f\j0);
          \draw[gray,->] (f\j2) -- (f\j1);
          \draw[gray,->] (f\j2) -- (f\j3);
          \draw[gray,->] (f\j3) -- (f\j4);
          \draw[gray,->] (f\j4) -- (f\j0);
        \end{scope}
      }
      \draw node[circle,fill,scale=.3,label=0:$w_{S,p}$] (f0) at (0,1) {};
      \draw[->] (f01) .. controls +(0,.8) and +(0,-.5) .. (f0);
      \draw[->] (f10) .. controls +(0,.8) and +(0,-.5) .. (f0);
      \draw[->] (f-14) .. controls +(0,1) and +(0,-.5) .. (f0);
    \end{scope}

    \begin{scope}[xshift=+0cm]
      \foreach \j in {-1,...,1} {
        \begin{scope}[xshift=-.6*\j cm]
          \draw[blue] node[circle,fill,scale=.3] (p\j0) at (0*360/5+90:.2cm) {};
          \draw[red] node[circle,fill,scale=.3] (p\j1) at (1*360/5+90:.2cm) {};
          \draw[green] node[circle,fill,scale=.3] (p\j2) at (2*360/5+90:.2cm) {};
          \draw[blue] node[circle,fill,scale=.3] (p\j3) at (3*360/5+90:.2cm) {};
          \draw[red] node[circle,fill,scale=.3] (p\j4) at (4*360/5+90:.2cm) {};
          \draw[gray,->] (p\j1) -- (p\j0);
          \draw[gray,->] (p\j2) -- (p\j1);
          \draw[gray,->] (p\j2) -- (p\j3);
          \draw[gray,->] (p\j3) -- (p\j4);
          \draw[gray,->] (p\j4) -- (p\j0);
        \end{scope}
      }
      \draw node[circle,fill,scale=.3,label=0:$w_{S,i}$] (p0) at (0,1) {};
      \draw[->] (p-11) .. controls +(0,.7) and +(0,-.5) .. (p0);
      \draw[->] (p00) .. controls +(0,.7) and +(0,-.5) .. (p0);
      \draw[->] (p12) .. controls +(0,1) and +(0,-.5) .. (p0);
    \end{scope}
  \end{tikzpicture}
  \vspace{-.7cm}
\end{center}
\caption{Construction of $D_{c+1}$}\label{fig:hyperedge}
\end{figure}
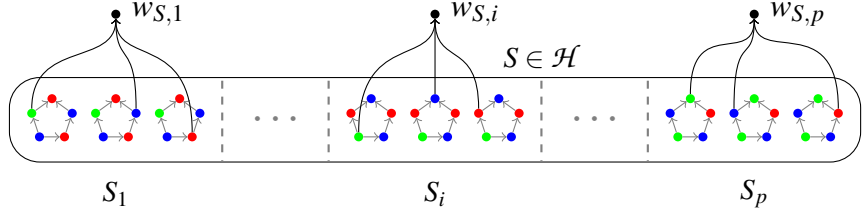

Firstly it is acyclic, as we only add sinks (the $w_{S,i}$) to disjoint copies
of $D_c$, which are acyclic by the induction hypothesis.

Secondly, every oriented cycle $C$ in $D_{c+1}$ has more than $b$ blocks. If $C$ is in a copy of $D_c$, then we have the result  by the induction hypothesis. Henceforth we may assume that $S$ contains some vertices in $W$, say
$w_1,...,w_{b'}$ in cyclic order around $C$. As the
vertices of $W$ are all sinks, the number of blocks of $C$ is at least $2b'$.
Let us denote by $S_{w_i}$ the hyperedge of $\mathcal H$ which triggered the creation of $w_i$. Then
two consecutive $S_{w_i},S_{w_{i+1}}$ (indices are modulo $b'$) have a
vertex $x_i$ of $X$ in common (indeed, the vertices between $w_i$ and $w_{i+1}$ in $C$ belong to some copy $D^i_c$). Therefore the sequence $x_{b'}, S_{w_1}, x_1, S_{w_2}, x_2, \dots , S_{w_{b'}}, x_{b'}$ contains a cycle  in ${\cal H}$. Hence by our choice of ${\cal H}$,  $b'>b/2$, so $C$ has more than $b$ blocks.


Finally, let us prove that $\chi(D_{c+1}) =c+1$.
We added a stable set to the disjoint union of copies of $D_c$, so $\chi(D_{c+1})\leq \chi(D_c) +1 =c+1$.\\
Now suppose for a contradiction that $D_{c+1}$ admits a proper $c$-colouring
$\phi$. It induces on ${\cal H}$ the $p$-colouring $\psi$ where $\psi(x_k)$ is
the index of the colouring of $D_c$ on $D_c^k$, i.e. the restriction of $\phi$
on $D_c^k$ is the colouring $col^{\psi(x_k)}_c$.  Now since ${\cal H}$ is
$(p+1)$-chromatic, there exists an hyperedge $S$ of ${\cal H}$ which is
monochromatic.  Let $i$ be the integer such that $\psi(x)=i$ for all $x\in S$.
Consider $S_i=\{x_{k_1},\dots, x_{k_c}\}$ and let $v_{k_1}\in D^{k_1}_c,\dots
,v_{k_c}\in D^{k_c}_c$ be the in-neighbours of $w_{S,i}$.  By construction,
$col^{i}_c(v_{k_1})=1,\dots ,col^{i}_c(v_{k_c})=c$, so $\phi(v_{k_1})=1 , \dots
, \phi(v_{k_c})=c$.  Consequently $w_{S,i}$ has the same colour (by $\phi$) as
one of its in-neighbours. This contradicts the fact that $\phi$ is proper.
Hence $\chi(D_{c+1}) \geq c+1$.
\end{proof}

Theorems~\ref{thm:infty1} and~\ref{thm:infty2} directly follow from Theorem~\ref{thm:blocks}, since a cycle and its subdivision have the same number of blocks.

\section{Cycles with two blocks in strong digraphs}\label{sec:2blocks}

In this section we first prove that $\SForb(C(k,\ell)) \cap {\cal S}$ has bounded chromatic number for every $k,\ell$. 
We need some preliminaries.

\subsection{Definitions and tools}

\subsubsection{Levelling}\label{subsubsec:level}

In a digraph $D$, the {\it distance} from a vertex $x$ to another $y$, denoted by $\dist_D(x,y)$ or simply  $\dist(x,y)$ when $D$ is clear from the context, is the minimum length of an $(x,y)$-dipath or $+\infty$ if no such dipath exists.
For a set $X\subseteq V(D)$ and vertex $y \in V(D)$, we define $\dist (X, y) =\min \{\dist (x,y) \mid x\in X\}$ and $\dist (y, X) =\min \{\dist (y,x) \mid x\in X\}$, and
for two sets $X,Y \subseteq V(D)$,  $\dist (X, Y) =\min \{\dist (x,y) \mid x\in X, y\in Y\}$.

An {\it out-generator} in a digraph $D$ is a vertex $u$ such that for any $x\in V(D)$, there is an $(u,x)$-dipath.
Observe that in a strong digraph every vertex is an out-generator.

Let $u$ be an out-generator of $D$.
For every nonnegative integer $i$,  the {\it $i$th level  from $u$} in $D$ is $L^u_i=\{v \mid \dist_D(u,v)=i\}$.
Because $u$ is an out-generator, $\bigcup_{i} L^u_i =V(D)$.
Let $v$ be a vertex of $D$, we set $\lvl^u(v)= \dist_D(u,v)$, hence $v \in L^u_{\lvl(v)}$.
In the following, the vertex $u$ is always clear from the context. Therefore, for sake of clarity, we drop the superscript $u$.

The definition immediately implies the following.
\begin{proposition}\label{prop:longueur}
Let $D$ be a digraph having an out-generator $u$.
If $x$ and $y$ are two vertices of $D$ with  $\lvl(y) > \lvl(x)$, then every $(x,y)$-dipath has length at least $\lvl(y) - \lvl(x)$.
\end{proposition}

Let $D$ be a digraph and $u$ be an out-generator of $D$. A {\it Breadth-First-Search Tree}  or {\it BFS-tree} $T$ with {\it root $u$}, is a sub-digraph of $D$ spanning $V(D)$ such that $T$ is an oriented tree and, for any $v \in V(D)$, $dist_T(u,v) =  \dist_D(u,v)$.
It is well-known that if $u$ is an out-generator of $D$, then there exist BFS-trees with root $u$.

Let $T$ be a BFS-tree with root $u$.
For any vertex $x$ of $D$, there is an unique $(u,x)$-dipath in $T$. The \textit{ancestors}
of $x$ are the vertices on this dipath. For an ancestor $y$ of $x$, we note $y \geq_T x$.
If $y$ is an ancestor of $x$, we denote by $T[y,x]$ the unique $(y,x)$-dipath in $T$.
For any two vertices $v_1$ and $v_2$, the \textit{least common ancestor} of $v_1$ and $v_2$
is the common ancestor $x$ of $v_1$ and $v_2$ for which $\lvl(x)$ is maximal. (This is well-defined since $u$ is an ancestor of all vertices.)

\subsubsection{Decomposing a digraph}

The {\it union} of two digraphs $D_1$ and $D_2$ is the digraph $D_1\cup
D_2$ with vertex set $V(D_1)\cup V(D_2)$ and arc set $A(D_1)\cup
A(D_2)$. Note that $V(D_1)$ and $V(D_2)$ are not necessarily disjoint.

The following lemma is well-known.
\begin{lemma}\label{lem:decomp}
Let $D_1$ and $D_2$ be two digraphs.
$\chi(D_1\cup D_2) \leq \chi(D_1)\times \chi(D_2)$.
\end{lemma}
\begin{proof}
Let $D=D_1 \cup D_2$.  For $i\in \{1,2\}$, let $c_i$ be a proper colouring of $D_i$ with $\{1, \dots , \chi(D_i)\}$. Extend $c_i$ to $(V(D), A(D_i))$  by assigning the colour $1$ to all vertices in $V_{3-i}$.
Now the function $c$ defined by $c(v)=(c_1(v), c_2(v))$ for all $v\in V(D)$ is a proper colouring of $D$ with colour set $\{1, \dots , \chi(D_1)\} \times \{1, \dots , \chi(D_2)\}$.
\end{proof}

\subsection{General upper bound}\label{subsec:2blocks-gen}

\begin{theorem}\label{thm:Ckk}
Let $k$ and $\ell$ be two positive integers such that $k\geq \max\{\ell,3\}$, and let $D$ be a digraph in $\SForb(C(k,\ell)) \cap {\cal S}$.  Then, $\chi(D) \leq \borne .$
\end{theorem}

\begin{proof}

  Since $D$ is strongly connected, it has an out-generator $u$. Let $T$ be a
  BFS-tree with root $u$. We define the following sets of arcs.

\begin{eqnarray*}
A_0 & = & \{xy \in A(D) \mid \lvl(x)=\lvl(y)\} ;\\
A_1 & =  & \{xy\in A(D) \mid 0< | \lvl(x)-\lvl(y) |< k+\ell-3;\\
A' & = & \{xy\in A(D) \mid  \lvl(x)-\lvl(y)\geq k+\ell-3\}.
\end{eqnarray*}

Since $k+\ell-3>0$ and there is no arc $xy$ with $\lvl(y)>\lvl(x)+1$, $(A_0, A_1, A')$ is a partition of $A(D)$.
Observe moreover that $A(T)\subseteq A_1$.
We further partition $A'$ into two sets $A_2$ and $A_3$,
where $A_2 =\{ xy \in A' \mid y \mbox{ is an ancestor of } x \mbox{ in } T\}$ and $A_3=A'\setminus A_2$.
Then $(A_0, A_1, A_2, A_3)$ is a partition of $A(D)$. Let $D_j=(V(D), A_j)$ for all $j\in \{0,1,2,3\}$.


\begin{claim}\label{claim:D0}
$\chi(D_0) \leq k+\ell-2$.
\end{claim}
\begin{subproof}
Observe that $D_0$ is the disjoint union of the $D[L_i]$ where $L_i=\{v \mid \dist_D(u,v)=i\}$.
Therefore it suffices to prove that $\chi(D[L_i]) \leq k+\ell-2$ for all non-negative integer $i$.

$L_0=\{u\}$ so the result holds trivially for $i=0$.

Assume now $i\geq 1$.
Suppose for a contradiction $\chi(D[L_i]) \geq k+\ell -1$.
Since $k\geq 3$, by Theorem~\ref{thm:2blocks}, $D[L_i]$ contains a copy $Q$ of $P^+(k-1,\ell-1)$.
Let $v_1$ and $v_2$ be the initial and terminal vertices of $Q$, and let $x$ be the least common ancestor of $v_1$ and $v_2$.
By definition, for $j\in \{1,2\}$, there exists a $(x,v_j)$-dipath $P_j$ in $T$. By definition of least common ancestor, $V(P_1)\cap V(P_2)=\{x\}$, $V(P_j)\cap L_i =\{v_j\}$, $j=1,2$, and both $P_1$ and $P_2$ have length at least $1$.
Consequently, $P_1\cup P_2\cup Q$ is a subdivision of $C(k,\ell)$, a contradiction.
\end{subproof}

\begin{claim}\label{claim:D1}
$\chi(D_1) \leq k+\ell-3$.
\end{claim}
\begin{subproof}
Let $\phi_1$ be the colouring of $D_1$ defined by $\phi_1(x)= \lvl(x)\pmod{k+\ell-3}$.
By definition of $D_1$, this is clearly a proper colouring of $D_1$.
\end{subproof}

\begin{claim}\label{claim:D2}
$\chi(D_2) \leq 2\ell+2$.
\end{claim}
\begin{subproof}
Suppose for a contradiction that $\chi(D_2)\geq 2\ell+3$. By Theorem~\ref{thm:2blocks},  $D_2$ contains a copy $Q$ of $P^-(\ell+1,\ell+1)$, which is the union of two disjoint dipaths which are disjoint except in there initial vertex $y$, say $Q_1 =(y_0, y_1,y_2, \dots ,y_{\ell+1})$ and $Q_2=(z_0, z_1, z_2, \dots ,z_{\ell+1})$ with $y_0=z_0=y$. Since $Q$ is in $D_2$,
all vertices of $Q$ belong to $T[u,y]$.  Without loss of generality, we can assume $z_{1} \geq_T y_{1}$.

If $z_{\ell+1} \geq_T y_{\ell+1}$, then let $j$ be the smallest integer such that $z_j\geq_T y_{\ell+1}$. Then the union of
$T[y_1,y] \odot Q_2[y,z_j]\odot T[z_j, y_{\ell+1}]$ and $Q_1[y_1,y_{\ell+1}]$ is a subdivision of $C(k,\ell)$, because $T[y_1,y]$ has length at least $k-2$ as $\lvl(y) \geq \lvl(y_1) +k +\ell -3$. This is a contradiction.

Henceforth $y_{\ell+1} \geq_T z_{\ell+1}$.
Observe that all the $z_j$, $1\leq j\leq \ell+1$ are in $T[y_{\ell+1} , y_1]$.
This, by the Pigeonhole principle, there exists $i,j\geq 1$ such that
$y_{i+1}\geq_T z_{j+1} \geq_T z_j\geq_Ty_i\geq_T z_{j-1}$.

If $\lvl(z_{j-1}) \geq \lvl(y_i)+\ell -1$, then $T[y_i, z_{j-1}]\odot (z_{j-1},z_j)$ has length at least $\ell$.
Hence its union with $(y_i, y_{i+1})\odot T[y_{i+1}, z_j]$, which has length greater than $k$, is a subdivision of $C(k,\ell)$, a contradiction.

Thus $\lvl(z_{j-1}) < \lvl(y_i)+\ell -1$ (in particular, in this case, $j>1$ and $i>2$). Therefore, by definition of $A'$, $\lvl(y_i) \geq \lvl (z_j) +k-1$ and $\lvl(y_{i-1}) \geq \lvl (z_{j-1}) +k-1$.
Hence both $T[z_{j-1}, y_{i-1}]$ and $T[z_{j}, y_{i}]$ have length at least $k-1$. So
 the union of $T[z_{j-1}, y_{i-1}]\odot (y_{i-1},y_i)$ and $(z_{j-1},z_j) \odot T[z_{j}, y_{i}]$ is a subdivision of $C(k,k)$ (and thus of $C(k,\ell)$), a contradiction.
\end{subproof}

\begin{claim}\label{claim:D3}
$\chi(D_3) \leq k+\ell +1$.
\end{claim}
\begin{subproof}
In this claim, it is important to note that $k+\ell-3\geq k-1$ because $\ell\geq 2$.
We use the fact that $\lvl(x) -\lvl(y) \geq k-1$ if $xy$ is an edge in $A_3$.

Suppose for a contradiction that $\chi(D_3) \geq k+\ell +1$.
By Theorem \ref{thm:2blocks}, $D_3$ contains a copy $Q$ of $P^-(k,\ell)$ which is the union of two disjoint dipaths which are disjoint except in there initial vertex $y$, say $Q_1 =(y_0, y_1,y_2, \dots ,y_{k})$ and $Q_2=(z_0, z_1, z_2, \dots ,z_{\ell})$ with $y_0=z_0=y$. 

Assume that a vertex of $Q_1-y$ is an ancestor of $y$.
Let $i$ be the smallest index such that $y_i$ is an ancestor of $y$.
If it exists, by definition of $A_3$, $i \geq 2$.
Let $x$ be the common ancestor of $y_i$ and $y_{i-1}$ in $T$.
By definition of $A_3$, $y_i$ is not an ancestor of $y_{i-1}$, so $x$ is different from $y_i$ and $y_{i-1}$.
Moreover  by definition of $A'$, $\lvl(y)-k\geq\lvl(y_{i-1}) -k \geq \lvl(y_i) -1 \geq \lvl(x)$.
Hence $T[x,y_{i-1}]$ and $T[x, y]$ have length at least $k$. Moreover these two dipaths are disjoint except in $x$.
Therefore, the union of $T[x,y_{i-1}]$ and $T[x,y]  \odot Q_1[y, y_{i-1}]$ is a subdivision of $C(k,k)$ (and thus of $C(k,\ell)$), a contradiction.

 Similarly, we get a contradiction if a vertex of $Q_2-y$ is an ancestor of $y$.
Henceforth, no vertex of $V(Q_1)\cup V(Q_2)\setminus \{y\}$ is an ancestor of $y$.

Let $x_1$ be the least common ancestor of $y$ and $y_1$. Note that $|T[x_1,y]| \geq k$ so $|T[x_1,y_1]| < k$, for otherwise
$G$ would contain a  subdivision of $C(k,k)$. Therefore $\lvl(y_1) - \lvl(x_1) < k$.
We define inductively $x_2, \dots, x_{k}$ as follows: $x_{i+1}$ is the least common ancestor of $x_i$ and $y_i$.
As above $|T[x_i,y_{i-1}]| \geq k$ so $\lvl(y_i) - \lvl(x_i)< k$.
Symmetrically, let $t_1$ be the least common ancestor of $y$ and $z_1$ and for $1\leq i\leq \ell-1$, let $t_{i+1}$ be the least common ancestor of $t_i$ and $z_i$. For $1\leq i\leq \ell$, we have $\lvl(z_i) - \lvl(t_i) < k$.
Moreover, by definition all $x_i$ and $t_j$ are ancestors of $y$, so they all are on $T[u,y]$.

Let $P_{y}$ (resp. $P_{z}$) be a shortest dipath in $D$ from $y_{k}$
(resp. $z_{\ell}$) to $T[u,y]  \cup Q_1[y_1, y_{k-1}] \cup  Q_2[z_1, z_{\ell -1}]$. Note that $P_y$ and $P_z$ exist since $D$ is strongly connected.
Let $y'$ (resp. $z'$) be the terminal vertex of $P_y$ (resp. $P_z$).
Let $w_y$ be the last vertex of $T[x_{k}, y_{k}]$  in $P_y$ (possibly, 
$w_y= y_{k}$.)  Similarly,  let $w_z$ be the last vertex of $T[t_{\ell}, z_{\ell}]$  in $P_z$ (possibly, 
$w_z= z_{\ell}$.)
Note that $P_y[w_y,y']$ is a shortest dipath from $w_y$ to $y'$ and  $P_z[w_z,z']$ is a shortest dipath from $w_z$ to $z'$.

If $y'=y_j$ for $0\leq j \leq k-1$, consider $R=T[x_{k}, w_y]\odot P_y[w_y,y_j]$ is an $(x_{k}, y_j)$-dipath. By Proposition~\ref{prop:longueur}, $R$ has length at least $k$ because $\lvl(y_j) -\lvl(x_{k}) \geq \lvl(y_j) -\lvl(y_{k})+1\geq k$.
Therefore the union of $R$ and $T[x_{k},y]  \cup Q_1[y, y_j]$ is a subdivision of $C(k,k)$, a contradiction.

Similarly, we get a contradiction if $z'$ is in $\{z_1, \dots ,z_{\ell-1}\}$.
Consequently, $P_y$ is disjoint from $Q_1[y, y_{k-1}]$ and $P_z$ is disjoint from $Q_2[y, z_{\ell-1}]$.

If $P_y$ and $P_z$ intersect in a vertex $s$. By the above statement,  $s\notin V(Q)\setminus \{y_{k}, z_{\ell}\}$.
Therefore the union of $Q_1 \odot P_y[y_{k},s]$ and $Q_2 \odot P_z[z_{\ell},s]$ is a subdivision of $C(k,\ell)$, a contradiction.
Henceforth $P_y$ and $P_z$ are disjoint.

Assume both $y'$ and $z'$ are in $T[u,y]$. If  $y' \geq_Tz'$, then the union of $Q_1 \odot P_y \odot
T[y',z']$ and $Q_2 \odot P_z$ form a subdivision of $C(k,\ell)$; and  if $z' \geq_Ty'$, then the union of $Q_1 \odot P_y$ and $Q_2 \odot P_z\odot T[z',y']$ form a subdivision of $C(k,\ell)$. This is a contradiction.

Henceforth a vertex among $y'$ and $z'$ is not in $T[u,y]$.
Let us assume that $y'$ is not in $T[u,y]$ (the case $z'\not\in T[u,y]$ is similar), and so
$y'=z_i$ for some $1\leq i\leq \ell-1$.
If $\lvl(y') \geq \lvl(x_{k}) +k$, then both $T[x_{k}, w_y]\odot P_y[w_y,y']$ and $T[x_{k}, y] \odot Q_2[y,z_i]$ have length at least $k$ by Proposition~\ref{prop:longueur}, so their union is a subdivision of $C(k,k)$, a contradiction.
Hence $\lvl(x_{k}) \geq \lvl(z_i)-k+1\geq \lvl(z_{\ell})\geq \lvl(t_{\ell}).$

If $z'=y_j$ for some $j$, then necessarily $\lvl(z') \geq \lvl(x_{k}) +k\geq \lvl(t_{\ell})+k$ and both $T[t_{\ell}, w_z]\odot P_z[w_z,z']$ and $T[t_{\ell}, y] \odot Q_1[y,y_j]$ have length at least $k$, so their union is a subdivision of $C(k,k)$, a contradiction.

Therefore $z'\in T[u,y]$.
The union of $T[t_{\ell}, z']$ and $T[t_{\ell}, w_z]\odot P_z[w_z,z']$ is not a subdivision of $C(k,k)$ so by Proposition~\ref{prop:longueur},
$\lvl(z')\leq \lvl(t_{\ell} )+k-1 \leq \lvl(z_{\ell}) +k-1 \leq \lvl(z_{\ell-1}).$

If $\lvl(z')\leq \lvl(x_{k})$, then the union of $Q_1$ and $Q_2\odot P_{z} \odot T[z', y_{k}]$ is a subdivision of $C(k,\ell)$, a contradiction.
Hence $\lvl(z') > \lvl(x_{k})$.
Therefore $\lvl(y')=\lvl(z_i) \leq \lvl(x_{k}) +k -1 \leq \lvl(z') +k-2 \leq \lvl(z_{\ell}) + 2k-3$, which implies that $i=\ell-1$ that is $y'=z_i=z_{\ell-1}$.
Now the union of  $[T[x_1,y_1]]\odot  Q_1[y_1,y_{k}] \odot P_y$  and $T[x_1,y]\odot Q_2[y, z_{\ell-1}]$ is a subdivision of $C(k,\ell)$, a contradiction. 
\end{subproof}

Claims~\ref{claim:D0},~\ref{claim:D1},~\ref{claim:D2}, and~\ref{claim:D3}, together with Lemma \ref{lem:decomp} yield the result.
\end{proof}

\subsection{Better bound for Hamiltonian digraphs}

We now improve on the bound of Theorem~\ref{thm:Ckk} in case of digraphs having a Hamiltonian directed cycle. Therefore we define
$$\phi(k,\ell) = \max\{\chi(D) \mid D\in \SForb(C(k,\ell)) ~\mbox{and $D$ has a Hamiltonian directed cycle}\}.$$

This section aims at proving that $\phi(k,k)\leq 6k-6$.

Let $D$ be a digraph and let $C=(v_1,\ldots,v_n,v_1)$  be a  Hamiltonian cycle in $D$ ($C$ may be directed or not).

For any $i,j \leq n$, let $d_C(v_i,v_j)$ be the distance between $v_i$ and $v_j$ in the undirected cycle $C$. That is, $d_C(v_i,v_j)=\min\{j-i,n-j+i\}$ if $j>i$ and $d_C(v_i,v_j)=\min\{i-j,n-i+j\}$ otherwise.

A  {\it chord} is an arc of $A(D)\setminus A(C)$. The {\it span} $\spa_C(a)$ of a chord $a=v_iv_j \in F$ is $d_C(i,j)$. We denote by $\spa_C(D)$ be the maximum span of a chord in $D$.


\begin{lemma}\label{lem:longChord}
If $D$ is a digraph with a Hamiltonian cycle $C$ and at least one chord, then $\chi(D) < 2\cdot  \spa_C(D)$.
\end{lemma}
\begin{proof}
Set $C=(v_1,\ldots,v_n,v_1)$ and set $\ell=\spa_C(D)$.
If $n<2\ell$, then the result trivially holds. Let us assume that $n=k \ell + r$ with $k\geq 2$ and $r < \ell$. Consider the following colouring. For any $1 \leq i\leq k \ell$, let us colour $v_i$ with colour $i-\left\lfloor i/\ell \right\rfloor \ell$. For any $1<t \leq r$, let us colour $v_{k\ell +t}$ with $\ell+t-1$. This colouring uses the  $\ell+r$ colours of $\{0, \dots , \ell+r-1\}$.

Moreover, for any $1\leq i \leq n$, all neighbours (in-neighbours and out-neighbours) of $v_i$ belong to $\{v_{i-\ell},\ldots,$ $v_{i-1}\} \cup \{v_{i+1},\ldots,v_{i+\ell}\}$ (all indices must be taken moduo $n$), for otherwise there would be a chord with span strictly larger than $\ell$.
Hence, the colouring is proper.
\end{proof}

Let $A \subseteq V(D)$, let $N(A) \subseteq V(D) \setminus A$ be the set of vertices not in $A$ that are adjacent to some vertex in $A$.

\begin{lemma}\label{lem:combine}
Let $D$ be a digraph and let $(A,B)$ be a partition of $V(D)$.
Then $$\chi(D)= \max \{ \chi(D[A]) + |N(A)|, \chi(D[B])\}.$$
\end{lemma}
\begin{proof}
Let us consider a proper colouring of $D[B]$ with colour set $\{1, \dots, \chi(D[B])\}$. W.l.o.g., vertices in $N(A)$ have received colours in $\{1,\ldots,|N(A)|\}$. Let us colour $D[A]$ using colours in $\{|N(A)|+1,\ldots,|N(A)|+\chi(D[A])\}$. We obtain a proper colouring of $D$ using $\max \{ \chi(D[A]) + |N(A)|, \chi(D[B])\}$ colours.
\end{proof}

\begin{lemma}\label{lem:neighbour}
Let $D$ be a  digraph containing no subdivision of $C(k,k)$ and having a Hamiltonian directed cycle $C=(v_1,\ldots,v_n,v_1)$.
Assume that $D$ contains a chord $v_iv_j$ with span at least $2k-2$ and let $A=\{v_{i+1},\ldots,v_{j-1}\}$ and $B=\{v_{j+1},\ldots,v_{i-1}\}$ (indices are taken modulo $n$). Then $|N(A)| \leq 2k+1$ and $|N(B)| \leq 2k+1$.
\end{lemma}

\begin{proof}
W.l.o.g., assume that $D$ has a chord $v_1v_j$ with $2k-1\leq j \leq n-2k+3$.


Assume first that $v_av_b$ is an arc from $A$ to $B$.
\begin{itemize}
\item[(1)] we cannot have $a\leq j-k$ and  $b \leq n-k+1$, for otherwise the two dipaths $C[v_a,v_j]$ and $(v_a,v_b)\odot C[v_b,v_1]\odot (v_1,v_j)$ have length at least $k$ and so their union is a subdivision of  $C(k,k)$, a contradiction.

\item[(2)] we cannot have $a\geq k$ and  $b \geq j+k-1$, for otherwise the two dipaths $C[v_1,v_a]\odot (v_a,v_b)$ and $(v_1,v_j) \odot C[v_j,v_b]$ have length at least $k$ and so their union is a subdivision of  $C(k,k)$, a contradiction.
\end{itemize}
Since $j\geq 2k-1$, either $a\leq j-k$ or $a\geq k$, so $v_b\in \{v_{j+1},\ldots,v_{j+k-2}\} \cup \{v_{n-k+2},\ldots,v_n\}$. Similarly, since $j\leq n-2k+3$, either
$b \leq n-k+1$ or  $b \geq j+k-1$, so $v_a\in \{v_2, \dots, v_{k-1}\} \cup \{v_{j-k+1} , \dots , v_{j-1}\}$.

Analogously, if  $v_bv_a$ is an arc from $B$ to $A$, we obtain that $v_a\in \{v_2, \dots, v_{k}\} \cup \{v_{j-k+2} , \dots , v_{j-1}\}$ and
$v_b\in \{v_{j+1},\ldots,v_{j+k-2}\} \cup \{v_{n-k+3},\ldots,v_n\}$.

Therefore $N(A)\subseteq \{v_1, \dots, v_k\} \cup \{v_{j-k+1} , \dots , v_{j}\}$, and $N(B)\subseteq \{v_j, \dots , v_{j+k-2}\} \cup \{v_{n-k+2}, \dots , v_n, v_1\}$.
Hence $|N(A)| \leq 2k+1$ and  $|N(B)| \leq 2k+1$.
\end{proof}

\begin{theorem}
Let $D$ be a digraph and let $k\geq 1$ be an integer.
If $D$ has a Hamiltonian directed cycle  and $\chi(D)>6k-6$, then $D$ contains a subdivision of a $C(k,k)$. In other words, $\phi(k,k)\leq 6k-6$.
\end{theorem}

\begin{proof}
If $k=2$, then we have the result by Theorem~\ref{thm:C(2,2)}. Henceforth, we assume $k\geq 3$.

For sake of contradiction, let us consider a counterexample (i.e a digraph $D$ with a Hamiltonian directed cycle, $\chi(D) > 6k-6$ and no subdivision of $C(k,k)$)
with the minimum number of vertices.

Let $C=(v_1,\ldots,v_n,v_1)$ be a Hamiltonian directed cycle of $D$.
By Lemma~\ref{lem:longChord} and because $\chi(D)\geq 4k-4$, $D$ contains a chord of span at least $2k-2$. Let $s$ be the minimum span of a chord of span at least $2k-2$ and consider a chord of span $s$. W.l.o.g., this chord is $v_1v_{s+1}$. Let $D_1=D[v_1,\ldots,v_{s+1}]$ and let $D_2=D[v_{s+1},\ldots, v_n, v_1]$. By minimality of the span of $v_1v_{s+1}$, either $D_1$ or $D_2$ contains no chord of span at least $2k-2$.
There are two cases to be considered.
\begin{itemize}
\item Assume first that $D_1$ contains no chord of span at least $2k-2$. By Lemma~\ref{lem:longChord}, $\chi(D_1)\leq 4k-7$. Let $A=\{v_2,\ldots,v_s\}$. We have $\chi(D[A])\leq \chi(D_1)\leq 4k-7$. Moreover, by Lemma~\ref{lem:neighbour}, $|N(A)|\leq 2k+1$.

Now  $D_2$ has a Hamiltonian directed cycle and contains no subdivision of $C(k,k)$. Therefore, $\chi(D_2)\leq 6k-6$ since $D$ has been chosen minimum.
Finally, by Lemma~\ref{lem:combine}, since $\chi(D[A]) + |N(A)| \leq 6k-6$ and  $\chi(D_2)\leq 6k-6$, we get that $\chi(D)\leq 6k-6$, a contradiction.

\item Assume now that $D_2$ contains no chord of span at least $2k-2$.
Set $B= \{v_{s+1}, \dots , v_n\}$. Similarly as in the previous case, we have $\chi(D[B])\leq \chi(D_2)\leq 4k-7$ and $|N(B)|\leq 2k+1$.

Let $D'_1$ be the digraph obtained from $D_1$ by reversing the arc $v_1v_s$. Clearly $D'_1$ is Hamiltonian.
Moreover, $D'_1$ contains no subdivision of a $C(k,k)$ ; indeed if it had such a subdivision $S$, replacing the arc $v_sv_1$ by $C[v_s,v_1]$ if it is in $S$, we obtain a subdivision of $C(k,k)$ in $D$, a contradiction.
Therefore $\chi(D_1)=\chi(D'_1)\leq 6k-6$,  by minimality of $D$.

Hence by Lemma~\ref{lem:combine}, since $\chi(D[B]) + |N(B)| \leq 6k-6$ and  $\chi(D_1)\leq 6k-6$, we get that $\chi(D)\leq 6k-6$, a contradiction.
\end{itemize}
\end{proof}

\subsection{Better bound when $\ell =1$}

We now improve on the bound of Theorem~\ref{thm:Ckk} when $\ell=1$.
To do so, reduce the problem to digraphs having a Hamiltonian directed cycle. 
Recall that $$\phi(k,\ell) = \max\{\chi(D) \mid D\in \SForb(C(k,\ell)) ~\mbox{and $D$ has a Hamiltonian directed cycle}\}.$$

\begin{theorem}\label{thm:k1}
Let $k$ be an integer greater than $1$.
$\chi(\SForb(C(k,1))\cap {\cal S})\leq \max\{2k-4, \phi(k,1)\}$.
\end{theorem}

To prove this theorem, we shall use the following lemma. 

\begin{lemma}\label{lem:k1}
Let $D$ be a digraph containing a directed cycle $C$ of length at least $2k-3$.
If there is a vertex $y$ in $V(D-C)$ and two distinct vertices $x_1,x_2\in V(C)$ such that for $i=1,2$, there is a $(x_i,y)$-dipath $P_i$ in $D$ with no internal vertices in $C$, then $D$ contains a subdivision of $C(k,1)$.
\end{lemma}
\begin{proof}
Since $C$ has length at least $2k-3$, then one of $C[x_1,x_2]$ and $C[x_2,x_1]$ has length at least $k-1$.
Without loss of generality, assume that $C[x_1,x_2]$ has length at least $k-1$.
Let $z$ be the first vertex along $P_2$ which is also in $P_1$.
Then the union of $C[x_1,x_2]\odot P_2[x_2, z]$ and $P_1[x_1,z]$ is a subdivision of $C(k,1)$.
\end{proof}


\begin{proof}[Proof of Theorem~\ref{thm:k1}]
Suppose for a contradiction that there is a strong digraph $D$ with chromatic number greater than $\max\{2k-4, \phi(k,1)\}$ that
contains no subdivision of $C(k,1)$. Let us consider the smallest such counterexample. 

All $2$-connected components of $D$ are strong, and one of them has chromatic number $\chi(D)$.
Hence, by minimality, $D$ is $2$-connected.
Let $C$ be a longest directed cycle in $D$.  By Bondy's theorem
(Theorem~\ref{thm:bondy}), $C$ has length at least $2k-3$, and by definition of $\phi(k,1)$, $C$ is not Hamiltonian.

Because $D$ is strong, there is a vertex $v\in C$ with an out-neighbour $w\not \in C$.
Since $D$ is $2$-connected, $D-v$ is connected, so there is a (not necessarily directed) oriented path in $D-v$ between $C-v$ and $w$.
Let $Q=(a_1, \dots , a_q)$ be  such a path so that all its vertices except the initial one are in $V(D)\setminus V(C)$. By definition $a_q=w$ and $a_1\in V(C)\setminus \{v\}$.
\begin{itemize}
\item Let us first assume that $a_1a_2 \in A(D)$. 
Let $t$ be the largest integer such that there is a dipath from $C-v$ to $a_t$ in $D-v$. Note that $t>1$ by the hypothesis.
If $t=q$, then by Lemma~\ref{lem:k1}, $C$ contains a subdivision of $C(k,1)$, a contradiction.
Henceforth we may assume that $t<q$. By definition of $t$,
$a_{t+1}a_t$ is an arc.
Let $P$ be a shortest $(v, a_{t+1})$-dipath in $D$. Such a dipath exists because $D$ is strong.
By maximality of $t$, $P$ has no internal vertex in $(C-v) \cup Q[a_1,a_t]$.
Hence, $a_t \in D-C$ and there are an $(a_1,a_t)$-dipath and a $(v,a_t)$-dipath with no internal vertices in $C$.
Hence, by Lemma~\ref{lem:k1}, $D$ contains a subdivision of $C(k,1)$, a contradiction.

\item Now, we may assume that any oriented path $Q=(a_1, \dots , a_q)$ from $C-v$ to $w$ starts with a backward arc, i.e., $a_2a_1 \in A(D)$. Let $W$ be the set of vertices $x$ such that there exists a (not necessarily directed) oriented path from $w$ to $x$ in $D-C$. In particular, $w \in W$. 

By the assumption, all arcs between $C - v$ and $W$ are from $W$ to $C -v$. Since $D$ is strong, this implies that, for any $x \in W$, there exists a directed $(w,x)$-dipath in $W$. In other words, $w$ is an out-generator of $W$.  Let $T_w$ be a BFS-tree of $W$ rooted in $w$  (see definitions in Section~\ref{subsubsec:level}). 

Because $D$ is strong and $2$-connected, there must be a vertex $y \in C -v$ such that there is an arc $ay$ from a vertex $a\in W$ to $y$. 

For purpose of contradiction, let us assume that there exists $z \in C -y$ such that there is an arc $bz$ from a vertex $b\in W$ to $z$. Let $r$ be the least common ancestor of $a$ and $b$ in $T_w$. If $|C[y,z]|\geq k$, then $T_w[r,a] \odot (a,y) \odot C[y,z]$ and $T_w[r,b] \odot (b,z) $ is a subdivision of $C(k,1)$. If $|C[z,y]|\geq k$, then $T_w[r,a] \odot (a,y)$ and $T_w[r,b] \odot (b,z) \odot C[z,y]$ is a subdivision of $C(k,1)$. In both cases, we get a contradiction.

From previous paragraph and the definition of $W$, we get that all arcs from $W$ to $D\setminus W$ are from $W$ to $y \neq v$, and there is a single arc from $D \setminus W$ to $W$ (this is the arc $vw$). 
Note that, since $D$ is strong, this implies that $D-W$ is strong.

Let $D_1$ be the digraph obtained from $D-W$ by adding the arc $vy$ (if it does not already exist). 
$D_1$ contains no subdivision of $C(k,1)$, for otherwise $D$ would contain one (replacing the arc $vy$ by the dipath $(v,w) \odot T_w[w,a] \odot (a,y)$). 
Since $D_1$ is  strong (because $D-W$ is strong), by minimality of $D$, $\chi(D_1) \leq \max\{2k-4, \phi(k,1)\}$. 

Let $D_2$ be the digraph obtained from $D[W\cup\{v,y\}]$ by adding the arc $yv$.
$D_2$ contains no subdivision of $C(k,1)$, for otherwise $D$ would contain one (replacing the arc $yv$ by the dipath $C[y,v]$). Moreover, $D_2$ is strong, so by minimality of $D$, $\chi(D_2) \leq \max\{2k-4, \phi(k,1)\}$. 

Consider now $D^*$ the digraph $D_1\cup D_2$. It is obtained from $D$ by adding the two arcs $vy$ and $yv$ (if they did not already exist).
Since $\{v,y\}$ is a clique-cutset in $D^*$, we get $\chi(D^*) \leq \max\{\chi(D_1),\chi(D_2)\} \leq \max\{2k-4, \phi(k,1)\}$. 
But $\chi(D) \leq \chi(D^*)$, a contradiction.
\end{itemize}
\end{proof}

From Theorem~\ref{thm:k1}, one easily derives an upper bound on $\chi(\SForb(C(k,1))\cap {\cal S})$.
\begin{corollary}\label{cor:premier}
$\chi(\SForb(C(k,1))\cap {\cal S})\leq 2k-1$.
\end{corollary}
\begin{proof}
By Theorem~\ref{thm:k1}, it suffices to prove $\phi(k,1)\leq 2k-1$.

Let $D\in \SForb(C(k,1))$ with a Hamiltonian directed cycle  $C=(v_1, \dots , v_n,v_1)$.
Observe that if  $v_iv_j$ is an arc, then $j\in C[v_{i+1}, v_{i+k-1}]$ for otherwise the union of $C[v_i,v_j]$ and $(v_i,v_j)$ would be a subdivision of $C(k,1)$.
In particular, every vertex had both its in-degree and out-degree at most $k-1$, and so degree at most $2k-2$.
As $\chi(D)\leq \Delta(D)+1$, the result follows.
\end{proof}

The bound $2k-1$ is tight for $k=2$, because of the directed odd cycles.
However, for larger values of $k$,  we can get  a better bound on $\phi(k,1)$, from which one derives a slightly better one for $\chi(\SForb(C(k,1))\cap {\cal S})$.

\begin{theorem}\label{thm:phik1}
$\phi(k,1)\leq \max\{k+1, \frac{3k-3}{2}\}$.
\end{theorem}

\begin{proof}
For $k = 2$, the result holds because  $\phi(2,1) \leq \phi(2,2)\leq 3$ by Corollary~\ref{cor:C(2,2)}.

Let us now assume $k \geq 3$.
We prove by induction on $n$, that every digraph $D\in \SForb(C(k,1))$ with a Hamiltonian directed cycle  $C=(v_1, \dots , v_n,v_1)$
has chromatic number at most $\max\{k+1, \frac{3k-3}{2}\}$, the result holding trivially when $n\leq \max\{k+1, \frac{3k-3}{2}\}$.

Assume now that $n\geq \max\{k+1, \frac{3k-3}{2}\} +1$
All the indices are modulo $n$.
Observe that if  $v_iv_j$ is an arc, then $j\in C[v_{i+1}, v_{i+k-1}]$ for otherwise the union of $C[v_i,v_j]$ and $(v_i,v_j)$ would be a subdivision of $C(k,1)$.
In particular, every vertex had both its in-degree and out-degree at most $k-1$.

Assume that $D$ contains a vertex  $v_i$ with in-degree $1$ or out-degree $1$. Then $d(v_i)\leq k$.
Consider $D_i$ the digraph obtained from $D-v_i$ by adding the arc $v_{i-1}v_{i+1}$.
Clearly, $D_i$ has a Hamiltonian directed cycle. Moreover is has no subdivision of $C(k,1)$ for otherwise, replacing the arc $v_{i-1}v_{i+1}$ by $(v_{i-1}, v_i, v_{i+1})$ if necessary, yields a subdivision of $C(k,1)$ in $D$.
By the induction hypothesis, $D_i$ a $\max\{k+1, \frac{3k-3}{2}\}$-colouring which can be extended to $v_i$ because $d(v_i)\leq k$.

Henceforth, we may assume that $\delta^-(D), \delta^+(D)\geq 2$.

\begin{claim}\label{claim:visions}
 $d^+(v_i)+d^-(v_{i+1})\leq 3k -n -3$ for all $i$.
\end{claim}
\begin{subproof}
Let $v_{i^+}$ be the first out-neighbour of $v_{i}$ along
$C[v_{i+2}, v_{i-1}]$ and let $v_{i^-}$ be the last in-neighbour of $v_{i+1}$ along $C[v_{i+3}, v_{i}]$. 
There are $d^+(v_i)-1$ out-neigbours of $v_i$ in $C[v_{i^+}, v_{i-1}]$ which all must be in $C[v_{i^+}, v_{i+k-1}]$ by the above observation. Therefore $i^+ \leq i+k-d^+(v_i)$. Similarly, $i^-\geq i-k+d^-(v_{i+1})$.
\begin{itemize}
\item if $v_i \in C[v_{i^-},v_{i^+}]$, $C[v_{i^-},v_{i^+}]$ has length $i^+-i^-\leq 2k - d^+(v_i) - d^-(v_{i+1})$.
Hence $C[v_{i^+}, v_{i^-}]$ has length at least $n-2k + d^+(v_i) + d^-(v_{i+1})$.
But the union of $(v_i, v_{i^+})\odot C[v_{i^+}, v_{i^-}]\odot (v_{i^-}, v_{i+1})$ and $(v_i, v_{i+1})$ is not a subdivision of $C(k,1)$, so
$C[v_{i^+}, v_{i^-}]$ has length at most $k-3$.
Hence, $k-3 \geq n -2k + d^+(v_i) + d^-(v_{i+1})$, so $d^+(v_i) + d^-(v_{i+1}) \leq 3k-n -3$.
\item otherwise, $v_{i^+}\in C[v_{i^-}, v_{i+1}]$ and $v_{i^-}\in C[v_{i}, v_{i^+}]$. 
Both $C[v_{i^-}, v_{i+1}]$ and $C[v_{i}, v_{i^+}]$ have length less than $k$ as $v_{i^-}v_{i+1}$ and $v_{i^-}v_{i+1}$ are arcs.
Moreover, the union of these two dipaths is $C$ and their intersection  contains the three distinct vertices $v_i$, $v_{i+1}$, $v_{i^-}$.
Consequently, $n=|C|\leq |C[v_{i^-}, v_{i+1}]| + |C[v_{i}, v_{i^+}]|-3\leq 2k-3$.
Let $v_{i_0}$ be the last out-neighbour of $v_i$ along $C[v_{i+2}, v_{i-1}]$.
All the out-neighbours of $v_i$ and all the in-neighbours of $v_{i+1}$ are in $C[v_i,v_{i_0}]$ which has length less than $k$ because
$v_iv_{i_0}$ is an arc. Hence $d^+(v_i)+d^-(v_{i+1})\leq k$, so $d^+(v_i)+d^-(v_{i+1})\leq 3k -n -3$ because $n\geq 2k-3$.
\end{itemize}
\vspace*{-24pt}
\end{subproof}

But $n\geq \frac{3k-1}{2}$, so by the above claim, $d^+(v_i)+d^-(v_{i+1})\leq \frac{3k-5}{2}$ for all $i$.

Summing these inequalities over all $i$, we get $\sum_{i=1}^n (d^+(v_i) + d^-(v_{i+1}) \leq \frac{3k-5}{2} \cdot n$.
Thus $\sum_{i=1}^n d(v_i)= \sum_{i=1}^n (d^+(v_i) + d^-(v_i)) \leq \frac{3k-5}{2} \cdot n$.
Therefore there exists an index $i$ such that $v_i$ has degree at most $\frac{3k-5}{2}$.
Consider the digraph $D_i$ defined above.
It is Hamiltonian and contains no subdivision of $C(k,1)$. 
By the induction hypothesis, $D_i$ has a $\max\{k+1, \frac{3k-3}{2}\}$-colouring which can be extended to $v$ because $d(v_i)\leq \frac{3k-5}{2}$.
\end{proof}

\begin{corollary}\label{cor:k1}
Let $k$ be an integer greater than $1$. Then $\chi(\SForb(C(k,1))\cap {\cal S})\leq \max\{k+1,2k-4\}$.
\end{corollary}
\begin{proof}
By Theorems~\ref{thm:k1} and \ref{thm:phik1},  $\chi(\SForb(C(k,1))\cap {\cal S})\leq \max\{2k-4, k+1, \frac{3k-3}{2}\} = \max\{k+1,2k-4\}$.
\end{proof}

\section{Small cycles with two blocks in strong digraphs}\label{sec:2blocks-small}

\subsection{Handle decomposition}

Let $D$ be a strongly connected digraph. A {\it handle} $h$ of $D$ is a directed path
$(s,v_1,\ldots,v_\ell,t)$ from $s$ to $t$ (where $s$ and $t$ may be identical)
such that:
\begin{itemize}
\item $d^-(v_i)=d^+(v_i)=1$, for every $i$, and
\item removing the internal vertices and arcs of $h$ leaves $D$ strongly connected.
\end{itemize}

The vertices $s$ and $t$ are the {\it endvertices} of $h$ while
the vertices $v_i$ are its {\it internal vertices}. The vertex $s$ is the {\it initial vertex} of $h$ and $t$ its {\it terminal vertex}.
The {\it length} of a handle is the number of its arcs, here $\ell+1$.
A handle of length $1$ is said to be {\it trivial}.

Given a strongly connected digraph $D$, a {\it handle decomposition} of $D$
starting at $v\in V(D)$ is a triple $(v,(h_i)_{1\le i \le
  p},(D_i)_{0\le i \le p})$, where $(D_i)_{0\le i \le p}$ is a
sequence of strongly connected digraphs and $(h_i)_{1\le i \le p}$ is a sequence
of handles such that:
\begin{itemize}
\item $V(D_0)=\{v\}$,
\item for $1\le i\le p$, $h_i$ is a handle of $D_i$ and $D_i$ is the
  (arc-disjoint) union of $D_{i-1}$ and $h_i$, and
\item $D=D_p$.
\end{itemize}

A handle decomposition is uniquely determined by $v$ and either
$(h_i)_{1\le i \le p}$, or $(D_i)_{0\le i \le p}$. The number of
handles $p$ in any handle decomposition of $D$
is exactly $|A(D)|-|V(D)|+1$. The value $p$ is also called
the \emph{cyclomatic number} of $D$. Observe that $p=0$
when $D$ is a singleton and $p=1$ when $D$ is a directed cycle.

A handle decomposition $(v,(h_i)_{1\le i \le
  p},(D_i)_{0\le i \le p})$ is {\it nice} if all handles except the first one $h_1$ have distinct endvertices (i.e., for any $1<i\leq p$, the initial and terminal vertices of $h_i$ are distinct).

A digraph is {\it robust} if it is $2$-connected and strongly connected.
  The following proposition is well-known (see \cite{BoMu08} Theorem~5.13).

 \begin{proposition}\label{prop:nice-decomp}
 Every robust digraph admits a nice handle decomposition.
 \end{proposition}

\begin{lemma}\label{lem:reduc}
Every strong digraph $D$ with $\chi(D)\geq 3$ has a robust subdigraph $D'$ with  $\chi(D')=\chi(D)$ and which is an oriented graph.
\end{lemma}
\begin{proof}
Let $D$ be a strong digraph $D$ with $\chi(D)\geq 3$.
Let $D'$ be a  $2$-connected components of $D$ with the largest chromatic number.
Each $2$-connected component of a strong digraph is strong, so $D'$ is strong.
Moreover, $\chi(D')=\chi(D)$ because the chromatic number of a graph is the maximum of the chromatic numbers of its $2$-connected components.
Now by Bondy's Theorem (Theorem~\ref{thm:bondy}), $D'$ contains a cycle $C$ of length at least $\chi(D')\geq 3$.
This can be extended into a handle decomposition $(v,(h_i)_{1\le i \le
  p},(D_i)_{0\le i \le p})$ of $D$ such that $D_1=C$.
  Let $D''$ be the digraph obtained from $D'$ by removing the arcs $(u,v)$ which are trivial handles  $h_i$ and such that $(v,u)$ is in $A(D_{i-1})$, we obtain an oriented graph $D''$ which is robust and with $\chi(D'')=\chi(D')=\chi(D)$. 
\end{proof}

\subsection{$C(1,2)$}

\begin{proposition}\label{prop:C(1,2)}
  A robust digraph containing no subdivision of $C(1,2)$ is a directed cycle.
\end{proposition}
\begin{proof}
Let $D$ be a robust digraph containing no subdivision of $C(1,2)$. 
Assume for a contradiction that a robust digraph  of $D$ is not a directed cycle.
By Proposition~\ref{prop:nice-decomp}, it contains a directed cycle $C$ and a nice handle $h_2$ from $u$ to $v$.
Now the union of $h_2$ and $C[u,v]$ is a subdivision of $C(1,2)$.
\end{proof}

\begin{corollary}\label{cor:C(1,2)}
$\chi(\SForb(C(1,2))\cap {\cal S}) =3$.
\end{corollary}
\begin{proof}
Lemma~\ref{lem:reduc}, Proposition~\ref{prop:C(1,2)}, and the fact that every directed cycles is $3$-colourable imply
$\chi(\SForb(C(1,2))\cap {\cal S}) \leq 3$.

The directed cycles of odd length have chromatic number $3$ and contain no subdivision of $C(1,2)$.
Therefore, $\chi(\SForb(C(1,2))\cap {\cal S}) =3$.
\end{proof}

\subsection{$C(2,2)$}

\begin{theorem}\label{thm:C(2,2)}
Let $D$ be a strong digraph. If $\chi(D) \geq 4$, then $D$ contains a subdivision of $C(2,2)$.
\end{theorem}

\begin{proof}
By Lemma~\ref{lem:reduc}, we may assume that $D$ is robust. 

By Proposition~\ref{prop:nice-decomp}, $D$ has a nice handle decomposition.
Consider a nice decomposition $(v,(h_i)_{1\le i \le
  p},(D_i)_{0\le i \le p})$ that maximizes the sequence $(\ell_1, \dots , \ell_p)$ of the length of the handles with respect to the lexicographic order.

Let $q$ be the largest index such that $h_q$ is not trivial.

\medskip

Assume first that $q\neq 1$. Let $s$ and $t$ be the initial and terminal vertex of $h_q$ respectively.
There is an $(s,t)$-path $P$ in $D_{q-1}$.
If $P=(s,t)$, let $r$ be the index of the handle containing the arc $(s,t)$.
Obviously, $r < q$. Now replacing $h_r$ by the handle $h'_r$ obtained from it by replacing the arc $(s,t)$ by $h_q$ and replacing $h_q$ by $(s,t)$, we obtain a nice handle decomposition contradicting the minimality of $(v,(h_i)_{1\le i \le
  p},(D_i)_{0\le i \le p})$.
Therefore $P$ has length at least $2$. So $P\cup h_q$ is a subdivision of $C(2,2)$.

\medskip

Assume that $q=1$, that is $D$ has a hamiltonian directed cycle $C$.
Let us call {\it chords} the arcs of $A(D)\setminus A(C)$.
Suppose that two chords $(u_1,v_1)$ and $(u_2,v_2)$ {\it cross}, that is $u_2\in C]u_1,v_1[$ and $v_2\in C]v_1,u_1[$.
Then the union of $C[u_1,u_2]\odot (u_2,v_2)$ and $(u_1,v_1)\odot C[v_1,v_2]$ forms a subdivision of $C(2,2)$.

If no two chords cross, then one can draw $C$ in the plane and all chords inside it without any crossing.
Therefore the graph underlying $D$ is outerplanar and has chromatic number at most~$3$.
\end{proof}

Since the directed odd cycles are in $\SForb(C(2,2))$ and have chromatic number $3$, Theorem~\ref{thm:C(2,2)} directly implies the following.
\begin{corollary}\label{cor:C(2,2)}
$\chi(\SForb(C(2,2))\cap {\cal S}) =3$.
\end{corollary}

\subsection{$C(1,3)$}

\begin{theorem}\label{thm:C(1,3)}
Let $D$ be a strong digraph. If $\chi(D) \geq 4$, then $D$ contains a subdivision of $C(1,3)$.
\end{theorem}

\begin{proof}
By Lemma~\ref{lem:reduc}, we may assume that $D$ is robust. Thus, by Proposition~\ref{prop:nice-decomp}, $D$ has a nice handle decomposition.
Consider a nice decomposition $(v,(h_i)_{1\le i \le
  p},(D_i)_{0\le i \le p})$ that maximizes the sequence $(\ell_1, \dots , \ell_p)$ of the length of the handles with respect to the lexicographic order.

Let $q$ be the largest index such that $h_q$ is not trivial.

\medskip

\noindent \underline{Case 1}: Assume first that $q\neq 1$. Let $s$ and $t$ be the initial and terminal vertex of $h_q$ respectively.
Since $D_{q-1}$ is strong, there is an $(s,t)$-dipath $P$ in $D_{q-1}$.
If $P=(s,t)$, let $r$ be the index of the handle containing the arc $(s,t)$.
Obviously, $r < q$. Now replacing $h_r$ by the handle $h'_r$ obtained from it by replacing the arc $(s,t)$ by $h_q$ and replacing $h_q$ by $(s,t)$, we obtain a nice handle decomposition contradicting the minimality of $(v,(h_i)_{1\le i \le
  p},(D_i)_{0\le i \le p})$.
Therefore $P$ has length at least $2$.
If either $P$ or $h_q$ has length at least $3$, then $P\cup h$ is a subdivision of $C(1,3)$.
Henceforth, we may assume that both $P$ and $h_q$ have length $2$. Set $P=(s,u,t)$ and $h=(s,x,t)$. Observe that $V(D)=V(D_{q-1})\cup \{x\}$.

Assume that $x$ has a neighbour $t'$ distinct from $s$ and $t$. By directional duality (i.e., up to reversing all arcs), we may assume that $x\ra t'$.
Considering the handle decomposition in which $h_q$ is replaced by $(s,x,t')$ and $(x,t')$ by $(x,t)$, we obtain that there is a dipath $(s,u',t')$ in $D_{q-1}$.
Now, if $u'=t$, then the union of $(s,x,t')$ and $(s,u,t,t')$ is a subdivision of $C(1,3)$.
Henceforth, we may assume that $t\notin \{s,u,u', t'\}$.
Since $D_{q-1}$ is strong, there is a dipath $Q$ from $t$ to $\{s,u,u', t'\}$, which has length at least one by the preceding assumption. Note that $x \notin Q$ since $Q$ is a dipath in $D_{q-1}$.
Whatever vertex of $\{s,u,u', t'\}$ is the terminal vertex $z$ of $Q$, we find a subdivision of $C(1,3)$:
\begin{itemize}
\item If $z=s$, then the union of $(x,t')$ and $(x,t)\odot Q \odot (s,u',t')$ is  a subdivision of $C(1,3)$;
\item If $z=u$, then the union of $(s,u)$ and $h_q\odot Q$ is  a subdivision of $C(1,3)$;
\item If $z=u'$, then the union of $(s,u')$ and $h_q\odot Q$ is  a subdivision of $C(1,3)$;
\item If $z=t'$, then the union of $(s,x,t')$ and $(s,u,t)\odot Q$ is  a subdivision of $C(1,3)$.
\end{itemize}
\medskip

\noindent \underline{Case 2}: Assume that $q=1$, that is $D$ has a hamiltonian directed cycle $C$.
Assume that two chords $(u_1,v_1)$ and $(u_2,v_2)$ cross. Without loss of generality, we may assume that the vertices $u_1$, $u_2$, $v_1$ and $v_2$ appear in this order along $C$.
Then the union of $C[u_2,v_1]$ and $(u_2,v_2)\odot C[v_2,u_1]\odot(u_1,v_1)$ forms a subdivision of $C(1,3)$.

If no two chords cross, then one can draw $C$ in the plane and all chords inside it without any crossing.
Therefore the graph underlying $D$ is outerplanar and has chromatic number at most $3$.
\end{proof}

Since the directed odd cycles are in $\SForb(C(1,3))$ and have chromatic number $3$, Theorem~\ref{thm:C(1,3)} directly implies the following.

\begin{corollary}\label{cor:C(1,3)}
$\chi(\SForb(C(1,3))\cap {\cal S}) =3$.
\end{corollary}

\subsection{$C(2,3)$}

\begin{theorem}
Let $D$ be a strong directed graph. If $\chi(D)\geq 5$, then $D$ contains a subdivision of $C(2,3)$.
\end{theorem}
\begin{proof}
By Lemma~\ref{lem:reduc}, we may assume that $D$ is a robust oriented graph.
Thus, by Proposition~\ref{prop:nice-decomp}, $D$ has a nice handle decomposition.
Let  ${\rm HD}=((h_i)_{1\le i \le
  p},(D_i)_{1\le i \le p})$ be a nice decomposition that maximizes the sequence $(\ell_1, \dots , \ell_p)$ of the length of the handles with respect to the lexicographic order. Recall that $D_i$ is strongly connected for any $1\leq i \leq p$. In particular, $h_1$ is a longest directed cycle in $D$.
  Let $q$ be the largest index such that $h_q$ is not trivial.
Observe that for all $i>q$, $h_i$ is a trivial handle by definition of $q$ and, for $i\leq q$, all handles $h_i$ have length at least $2$.

  \begin{claim}
 For any $1<i\leq q$, $h_i$ has length exactly $2$.
 \end{claim}
 \begin{subproof}
 For sake of contradiction, let us assume that there exists $2\leq r \leq q$ such that  $h_r=(x_1,\ldots,x_t)$ with $t\geq 4$. Since $D_{r-1}$ is strong, there is a $(x_1,x_t)$-dipath $P$ in $D_{r-1}$. Note that $P$ does not meet $\{x_2,\dots,x_{t-1}\}$. If $P$ has length at least $2$, then $P\cup h_r$ is a subdivision of $C(2,3)$. If $P=(x_1,x_t)$, let $r'$ be the handle containing the arc $h_{r'}$. Now the handle decomposition obtained from ${\rm HD}$ by replacing $h_{r'}$by the handle derived from it  by replacing the arc $(x_1,x_t)$ by $h_r$, and replacing $h_r$ by $(x_1,x_t)$, contradicts the maximality of ${\rm HD}$.   \end{subproof}

 For $1<i\leq q$, set $h_i=(a_i,b_i,c_i)$.
Since $h_1$ is a longest directed cycle in $D$ and $\chi(D) \geq 5$,  by Bondy's Theorem, $h_1$ has length at least $5$.
Set $h_1=(u_1, \dots , u_m, u_1)$.

A {\it clone} of $u_i$ is a vertex whose unique out-neighbour in $D_q$ is $u_{i+1}$ and whose unique in-neighbour in $D_q$ is $u_{i-1}$ (indices are taken modulo $m$).

\begin{claim}
Let $v \in V(D)\setminus V(D_1)$. Let $1<i \leq q$ such that $v=b_i$, the internal vertex of $h_i$.
There is an index $j$ such that $b_i$ is a clone of $u_j$, that is $a_i=u_{j-1}$ and $c_i=u_{j+1}$.
\end{claim}
\begin{subproof}
We prove the result by induction on $i$.

By the induction hypothesis (or trivially if $i=2$), there exists $i^-$ and $i^+$ such that
$a_i$ is $u_{i^-}$ or a clone of $u_{i^-}$ and $c_i$ is $u_{i^+}$ or a clone of $u_{i^+}$.
If $i^+\notin\{i^-+1, i^-+2\}$, then the union of $h_i$ and $(a_i, u_{i^-+1}, \dots , u_{i^+-1}, c_i)$ is a subdivision of $C(2,3)$, a contradiction
If $i^+=i^--1$, then $(a_i, b_i, c_i, h_1[u_{i^++1} , \dots , u_{i^--1}], a_i)$ is a cycle longer than $h_1$, a contradiction.
Henceforth $i^+=i^-+2$.
If $c_i$ is not $u_{i^+}$, then it is a clone of $u_{i^+}$. Thus the union of $(a_i,b_i, c_i, u_{i^++1})$ and $(a_i, u_{i^-+1}, u_{i^+}, u_{i^++1})$ is a subdivision of $C(2,3)$, a contradiction
Similarly, we obtain a contradiction if $a_i \neq u_{i^-}$.
Therefore, $a_i=u_{i^--1}$ and $c_i=u_{i^-+1}$, that is  $b_i$ is a clone of $u_{i^-+1}$.
Moreover all $b_{i'}$ for $i'<i$ are not adjacent to $b_i$ and thus are still clones of some $u_j$.
\end{subproof}

For $1\leq i\leq m$, let $S_i$ be the set of clones of $u_i$.
\begin{claim}\label{claim:clone2}
\item[(i)] If $S_i\neq \emptyset$, then $S_{i-1}=S_{i+1}=\emptyset$.
\item[(ii)] If $x\in S_i$, then $N^+_D(x)=\{u_{i+1}\}$ and $N^-_D(x)=\{u_{i-1}\}$.
\end{claim}
\begin{subproof}
(i) Assume for a contradiction, that both $S_{i}$ and $S_{i+1}$ are non-empty, say $x_i\in S_{i}$ and $x_{i+1}\in S_{i+1}$.
Then the union of $(u_{i-1}, u_i, x_{i+1}, u_{i+2})$ and $(u_{i-1}, x_i, u_{i+1}, u_{i+2})$ is a subdivision of $C(2,3)$, a contradiction.

\medskip

(ii) Let $x\in S_i$.
Assume for a contradiction that $x$ has an out-neighbour $y$ distinct from $u_{i+1}$. By (i), $y\notin S_{i-1}$, and $y\neq u_{i-1}$ because $D$ is an oriented graph.
If $y\in S_i\cup\{u_i\}$, then $(x, y, h_1[u_{i+1}, u_{i-1}], x)$ is a directed cycle longer than $h$.
If $y\in S_j\cup \{u_j\}$ for $j\notin \{i-2\}$, then the union of  $(u_{i-1}, x, y, u_{j+1})$ and $h_1[u_{i-1}, u_{j+1}]$ is a subdivision of $C(2,3)$, a contradiction.
If $y\in S_{i-2}$, then the union of  $(x, y, u_{i-1})$ and $(x, h_1[u_{i+1}, u_{i-1}]$ is a subdivision of $C(2,3)$, a contradiction.
If $y=u_j$ for $j\notin \{i-1, i, i+1\}$, then the union of  $(u_{i-1}, x, y)$ and $h_1[u_{i-1}, y]$ is a subdivision of $C(2,3)$, a contradiction.
\end{subproof}

This implies that $q=1$.
Indeed, if $q\geq 2$, then there is $i\leq m$ such that $b_2 \in S_i$.
But $D-b_q=D_{q-1}$ is strong, and $\chi(D-b_q) \geq 5$, because $\chi(D)\geq 5$ and $b_q$ has only two neighbours in $D$ by
Claim~\ref{claim:clone2}-(ii). But then by minimality of $D$, $D-b_q$ contains a subdivision of $C(2,3)$, which is also in $D$, a contradiction.

Hence $m=|V(D)|$.
Because $\chi(D)\geq 5$, $D$ is not outerplanar, so there must be $i<j<k<\ell <i +m$ such that $(u_i,u_k) \in A(D)$ and $(u_j,u_{\ell}) \in A(D)$. We must have $j=i+1$ and $\ell=k+1$ since otherwise $(u_i,\dots,u_j, u_{\ell})$ and $(u_i,u_k,\dots,u_{\ell})$ form a subdivision of $C(2,3)$. In addition, $k=j+1$ since otherwise, $(u_j,u_{\ell},\dots,u_i,u_k)$ and $(u_j,\dots,u_k)$ form a subdivision of $C(2,3)$. Therefore, any two ``crossing" arcs must have their ends being consecutive  in $D_1$. This implies that $N^+(u_j)=\{u_{j+1},u_{j+2}\}$, $N^-(u_j)=\{u_{j-1}\}$, $N^+(u_k)=\{u_{k+1}\}$ and $N^-(u_k)=\{u_{k-1},u_{k-2}\}$.

 Now let $D'$ be the digraph obtained from $D-\{u_j,u_k\}$ by adding the arc $(u_i,u_{\ell})$. Because $u_j$ and $u_k$ have only three neighbours in $D$, $\chi(D')\geq 5$. By minimality of $D$, $D'$ contains a subdivision of $C(2,3)$, which can be transformed into a subdivision of $C(2,3)$ in $D$ by replacing the arc $(u_i,u_{\ell})$ by the directed path $(u_i,u_j, u_k, _l)$.
\end{proof}

Since every semi-complete digraph of order $4$ does not contain $C(2,3)$ (which has order 5), we have the following.
\begin{corollary}\label{cor:C(2,3)}
$\chi(\SForb(C(2,3))\cap {\cal S}) =4$.
\end{corollary}


\section{Cycles with four blocks  in strong digraphs}\label{sec:4block}

\begin{theorem}\label{thm:hatC}
Let $D$ be a digraph in $\SForb(\hat{C}_4)$.
If $D$ admits an out-generator, then $\chi(D) \leq 24$.
\end{theorem}

\begin{proof}
The general idea is the same as in the proof of Theorem~\ref{thm:Ckk}.

Suppose that $D$ admits an out-generator $u$ and let $T$ be an BFS-tree with root $u$ (See Subsubsection~\ref{subsubsec:level}.).
We partition $A(D)$ into three sets according to the levels of $u$.
\begin{eqnarray*}
A_0 & = & \{(x,y) \in A(D) \mid \lvl(x)=\lvl(y)\} ;\\
A_1 & =  & \{(x,y) \in A(D)\mid  | \lvl(x)-\lvl(y) | =1\} ;\\
A_2 & = & \{(x,y) \in A(D) \mid \lvl(y) \leq \lvl(x)-2\}.
\end{eqnarray*}

For $i=0,1,2$, let $D_i=(V(D), A_i)$.

\begin{claim}\label{claim:A_0}
$\chi(D_0)\leq 3$.
\end{claim}
\begin{subproof}
Suppose for a contradiction that $\chi(D)\geq 4$. By Theorem~\ref{thm:2blocks}, it contains a $P^-(1,1)$ $(y_1,y,y_2)$, that is $y,y_1$ and $y,y_2$ are in $A(D_0)$. Let $x$ be the least common ancestor of $y_1$ and $y_2$ in $T$. The union of $T[x,y_1]$, $(y,y_1)$,
$(y,y_2)$, and $T[x,y_2]$ is a subdivision of $\hat{C}_4$, a contradiction.
\end{subproof}

\begin{claim}\label{claim:A_1}
$\chi(D_1)\leq 2$.
\end{claim}
\begin{subproof}
Since the arc are between consecutive levels, then the colouring $\phi_1$ defined by  $\phi_1(x) = lvl(x) \mod 2$ is a proper $2$-colouring of $D_1$.
\end{subproof}

Let $y\in V_i$ we denote by $N'(y)$ the out-degree of $y$ in $\bigcup_{0\leq j\leq i-1} V_j$.
Let $D' = (V,A')$ with $A' = \cup_{x\in V}\{(x,y), y\in N'(x)\}$ and $D_x = (V, A_x)$ where $A_x$ is the set of arc inside the level and from $V_i$ to $V_{i+1}$ for all $i$. Note that $A = A' \cup A_x$ and

\begin{claim}\label{claim:A_2}
$\chi(D_2)\leq 4$.
\end{claim}
\begin{subproof}
Let $x$ be a vertex of $V(D)$.
If $y$ and $z$ are distinct out-neighbours of $x$ in $D_2$, then
their least common ancestor $w$ is either $y$ or $z$, for otherwise the union of $T[w,y]$, $(x,y)$, $(x,z)$, and $T[w,z]$ is a subdivision of $\hat{C}_4$.
Consequently, there is an ordering $y_1,\dots, y_p$ of $N^+_{D_2}(x)$ such that the $y_i$ appear in this order on $T[u,x]$.

Let us prove that $N^+(y_i)=\emptyset$ for $2\leq i\leq p-1$.
Suppose for a contradiction that $y_i$ has an out-neighbour $z$ in $D_2$.
Let $t$ be the least common ancestor of $y_1$ and $z$.
If $t=z$, then the union of $(y_i,z)\odot T[z, y_1]$, $(x,y_1)$, $(x,y_p)$, and $T[y_i,y_p]$ is a subdivision of $\hat{C}_4$;
if $t=y\neq z$, then the union of $(y_i,z)$, $(x,y_1)\odot T[y_1,z]$, $(x,y_p)$, and $T[y_i,y_p]$ is a subdivision of $\hat{C}_4$. Otherwise, if $t \notin \{y,z\}$, $T[t,y_1]$, $T[t,z]$, $(x,y_i) \odot (y_i,z)$ and $(x,y_1)$ is a subdivision of $\hat{C}_4$.

Henceforth, in $D_2$, every vertex has at most two out-neighbours that are not sinks.
Let $V_0$ be the set of sinks in $D_2$. It is a stable set in $D_2$.
Furthermore $\Delta^+(D_2-V_0)\leq 2$, so $D_2-V_0$ is $3$-colourable, because $D_2$ (and so $D_2-V_0$) is acyclic.
Therefore $\chi(D_2)\leq 4$.
\end{subproof}

Claims~\ref{claim:A_0},~\ref{claim:A_1},~\ref{claim:A_2}, and Lemma \ref{lem:decomp} implies $\chi(D) \leq 24$.
\end{proof}

\section{Further research}

\medskip

The upper bound of Theorem~\ref{thm:Ckk} can be lowered when considering $2$-strong digraphs.

\begin{theorem}\label{thm:2strong}
Let $k$ and $\ell$ be two integers such that, $k \geq \ell$,  $k+\ell\geq 4$ and $(k,\ell)\neq (2,2)$.
Let $D$ be a $2$-strong digraph. If $\chi(D) \geq (k+\ell-2)(k-1)+2$, then $D$ contains a subdivision of $C(k,\ell)$.
\end{theorem}
\begin{proof}
Let $D$ be a $2$-strong digraph with chromatic number at least $(k+\ell-2)(k-1)+2$.
Let $u$ be a vertex of $D$.
For every positive integer $i$, let $L_i=\{v \mid \dist_D(u,v)=i\}$.

Assume first that $L_k\neq \emptyset$.
Take $v\in L_k$.
In $D$, there are two internally disjoint $(u,v)$-dipaths $P_1$ and $P_2$.
Those two dipaths have length at least $k$ (and $\ell$ as well) since $\dist_D(u,v)\geq k$.
Hence $P_1\cup P_2$ is a subdivision of $C(k,\ell)$.

Therefore we may assume that $L_k$ is empty, and so $V(D)=\{u\} \cup L_1 \cup \cdots \cup L_{k-1}$.
Consequently, there is $i$ such that $\chi(D[L_i])\geq k+\ell-1$.
Since $k+l-1\geq 3$ and $(k-1, \ell-1)\neq (1,1)$, by Theorem~\ref{thm:2blocks}, $D[L_i]$ contains a copy $Q$ of $P^+(k-1, \ell-1)$.
Let $v_1$ and $v_2$ be the initial and terminal vertices of $Q$.
By definition, for $j\in \{1,2\}$, there is a $(u,v_j)$-dipath $P_j$ in $D$ such that $V(P_j)\cap L_i =\{v_j\}$.
Let $w$ be the last vertex along $P_1$ that is in $V(P_1)\cap V(P_2)$.
Clearly, $P_1[w,v_1]\cup P_2[w,v_2]\cup Q$ is a subdivision of $C(k,\ell)$.
\end{proof}

To go further, it is natural to ask what happens if we consider digraphs which are not only strongly connected but $k$-strongly connected ($k$-strong for short).

\begin{proposition}\label{prop:existe}
Let $C$ be an oriented cycle of order $n$.
Every $(n-1)$-strong digraph contains a subdivision of $C$.
\end{proposition}
\begin{proof}
Set $C=(v_1,v_2, \dots , v_n,v_1)$. Without loss of generality, we may assume that $(v_1,v_n) \in A(C)$.
Let $D$ be an $(n-1)$-strong digraph.
Choose a vertex $x_1$ in $V(D)$. Then for $i=2$ to $n$, choose a vertex $x_i$ in $V(D)\setminus \{x_1, \dots, x_{i-1}\}$ such that
$x_{i-1}x_i$ is an arc in $D$ if $v_{i-1}v_i$ is an arc in $C$ and $x_ix_{i-1}$ is an arc in $D$ if $v_iv_{i-1}$ is an arc in $C$. This is possible since every vertex has in- and out-degree at least $n-1$.
Now, since $D$ is $(n-1)$-strong, $D - \{x_2, \dots, x_{n-1}\}$ is strong, so there exists a $(x_1, x_n)$-dipath $P$ in $D - \{x_2, \dots, x_{n-1}\}$.
The union of $P$ and $(x_1,x_2,\dots ,x_n)$ is a subdivision of $C$.
\end{proof}

Let ${\cal S}_p$ be the class of $p$-strong digraphs.
Proposition~\ref{prop:existe} implies directly that $\SForb(C)\cap {\cal S}_p=\emptyset$ and so  $\chi(\SForb(C)\cap {\cal S}_p)=0$ for any oriented cycle $C$ of length $p+1$.
This yields the following problems.

\begin{problem}
Let $C$ be an oriented cycle and $p$ a positive integer.
What is  $\chi(\SForb(C)\cap {\cal S}_{p})$ ?
\end{problem}

Note that $\chi(\SForb(C)\cap {\cal S}_{p+1}) \leq \chi(\SForb(C)\cap {\cal S}_{p})$ for all $p$, because ${\cal S}_{p+1} \subseteq {\cal S}_{p}$.

\begin{problem}
Let $C$ be an oriented cycle.
\begin{itemize}
\item[1)] What is the minimum integer $p_C$ such that $\chi(\SForb(C)\cap {\cal S}_{p_C})<+\infty$ ?
\item[2)] What is the minimum integer $p^0_C$ such that $\chi(\SForb(C)\cap {\cal S}_{p^0_C})=0$ ?
\end{itemize}
\end{problem}

\appendix


\begin{thebibliography}{XX}

\bibitem{AHT07}
L.~Addario-Berry, F.~Havet, and S.~Thomass{\'e}.
\newblock Paths with two blocks in $n$-chromatic digraphs.
\newblock \emph{Journal of Combinatorial Theory, Series B}, 97
  (4): 620--626, 2007.

\bibitem{AHS+13}
L.~Addario-Berry, F.~Havet, C.~L. Sales, B.~A. Reed, and S.~Thomass{\'e}.
\newblock Oriented trees in digraphs.
\newblock \emph{Discrete Mathematics}, 313 (8): 967--974,
  2013.


\bibitem{alon2014coloring}
N. Alon, A. Kostochka, B. Reiniger, D.~B West, and X.
  Zhu.
\newblock Coloring, sparseness, and girth.
\newblock {\em arXiv preprint arXiv:1412.8002}, 2014.



   \bibitem{Bon76} J. A. Bondy, Disconnected orientations and a conjecture of Las Vergnas,
{\it J. London Math. Soc. (2)}, {\bf 14} (2) (1976), 277--282.


  \bibitem{BoMu08}
J.A. Bondy and U.S.R. Murty.
\newblock {\em {G}raph {T}heory}, volume 244 of {\em Graduate Texts in
  Mathematics}.
\newblock Springer, 2008.



\bibitem{Burr80}
S.~A. Burr.
\newblock Subtrees of directed graphs and hypergraphs.
\newblock In \emph{Proceedings of the 11th Southeastern Conference on
  Combinatorics, Graph theory and Computing}, pages 227--239, Boca Raton - FL,
  1980. Florida Atlantic University.

\bibitem{Bur82} S. A. Burr,
Antidirected subtrees of directed graphs.
{\it Canad. Math. Bull.} {\bf 25} (1982), no. 1, 119--120.


\bibitem{Erd59}
P.~Erd{\H{o}}s.
\newblock Graph theory and probability.
\newblock {\em Canad. J. Math.}, 11:34--38, 1959.

\bibitem{ErHa66}
P.~Erd{\H{o}}s and A. Hajnal.
\newblock On chromatic number of graphs and set-systems.
\newblock {\em Acta Mathematica Academiae Scientiarum Hungarica}, 17(1-2):61--99, 1966.


\bibitem{EL75}
P.~Erd{\H{o}}s and L.~Lov{\'a}sz.
\newblock Problems and results on {$3$}-chromatic hypergraphs and some related
  questions.
\newblock In {\em Infinite and finite sets ({C}olloq., {K}eszthely, 1973;
  dedicated to {P}. {E}rd{\H o}s on his 60th birthday), {V}ol. {II}}, pages
  609--627. Colloq. Math. Soc. J\'anos Bolyai, Vol. 10. North-Holland,
  Amsterdam, 1975.


\bibitem{Gal68}
T.~Gallai.
\newblock On directed paths and circuits.
\newblock In \emph{Theory of Graphs (Proc. Colloq. Titany, 1966)}, pages
  115--118. Academic Press, New York, 1968.

\bibitem{Gya92}
A. Gy\'arf\'as.
\newblock Graphs with $k$ odd cycle lengths.
\newblock {\it Discrete Math.}, 103, pp. 41--48, 1992.


\bibitem{Has64}
M.~Hasse.
\newblock Zur algebraischen bergr\"und der graphentheorie {I}.
\newblock \emph{Math. Nachr.}, 28: 275--290, 1964.

\bibitem{HoTa73}
J. Hopcroft and R. Tarjan.
\newblock Efficient algorithms for graph manipulation.
\newblock {\it Communications of the ACM}, 16 (6): 372--378, 1973.

\bibitem{KRS11}
T. Kaiser, O. Ruck\'y, and R. Skrekovski.
\newblock Graphs with odd cycle lengths 5 and 7 are 3-colorable.
\newblock {\it SIAM J. Discrete Math.},25(3):1069--1088, 2011.


\bibitem{LRS10}
C. L\"owenstein, D. Rautenbach, and I. Schiermeyer.
\newblock Cycle length parities and the chromate number.
\newblock {\it J. Graph Theory}, 64(3):210--218, 2010.


\bibitem{MiSc04}
P. Mih\'ok and I. Schiermeyer.
\newblock Cycle lengths and chromatic number of graphs.
\newblock {\it Discrete Math.}, 286(1-2): 147--149, 2004.

\bibitem{Roy67}
B.~Roy.
\newblock Nombre chromatique et plus longs chemins d'un graphe.
\newblock \emph{Rev. Francaise Informat. Recherche Op\'erationnelle},
  1 (5): 129--132, 1967.


\bibitem{Sum81}
D.~P. Sumner.
\newblock Subtrees of a graph and the chromatic number.
\newblock In {\em The theory and applications of graphs (Kalamazoo, Mich.,
  1980)}, pages 557--576. Wiley, New York, 1981.

\bibitem{Vit62}
L.~M. Vitaver.
\newblock Determination of minimal coloring of vertices of a graph by means of
  boolean powers of the incidence matrix.
\newblock \emph{Doklady Akademii Nauk SSSR}, 147: 758--759, 1962.


\bibitem{Wan08}
S.S. Wang.
\newblock Structure and coloring of graphs with only small odd cycles.
\newblock {\it SIAM J.Discrete Math.}, 22:1040--1072, 2008.


\end{thebibliography}
\end{document}